\documentclass[11pt]{amsart}
\usepackage{graphicx}
\usepackage[dvips]{epsfig}
\usepackage{pinlabel}
\usepackage{amsmath}
\usepackage{amsfonts}
\usepackage{latexsym}
\usepackage{amssymb}
\usepackage[usenames]{color}
\usepackage{amsthm}
\usepackage[all]{xypic}
\usepackage{enumitem}
\usepackage[breaklinks=true]{hyperref}

\input xy
\xyoption{all}

\author{Baptiste Chantraine}
\author{Georgios Dimitroglou Rizell}
\author{Paolo Ghiggini}
\author{Roman Golovko}

\address{Universit\'e de Nantes, France.}
\email{baptiste.chantraine@univ-nantes.fr}

\address{University of Cambridge, United Kingdom.}
\email{g.dimitroglou@maths.cam.ac.uk}

\address{Universit\'e de Nantes, France.}
\email{paolo.ghiggini@univ-nantes.fr}

\address{Alfr\'{e}d R\'{e}nyi Institute of Mathematics, Hungary.}
\email{golovko.roman@renyi.mta.hu}

\theoremstyle{remark} \theoremstyle{definition}
\newtheorem{defn}{Definition}[section]
\newtheorem{Rem}[defn]{Remark}

\theoremstyle{plain}
\newtheorem{Thm}[defn]{Theorem}
\newtheorem{Prop}[defn]{Proposition}
\newtheorem{Lem}[defn]{Lemma}
\newtheorem{Cor}[defn]{Corollary}

\def\co{\colon\thinspace}

\DeclareMathAlphabet{\mathdj}{U}{msb}{m}{n}
\newcommand{\R}{\ensuremath{\mathdj{R}}}

\newcommand{\Z}{\ensuremath{\mathdj{Z}}}

\newcommand{\C}{\ensuremath{\mathdj{C}}}

\newcommand{\id}{\operatorname{Id}}

\begin{document}
\title{Floer homology and Lagrangian concordance}
\thispagestyle{empty}
\date{\today}
\thanks{The first author is partially supported by the ANR project COSPIN (ANR-13-JS01-0008-01). The second author is supported by the grant KAW 2013.0321 from the Knut and Alice Wallenberg Foundation. The fourth author is supported by the ERC Advanced Grant LDTBud.}
\subjclass[2010]{Primary 53D12; Secondary 53D42}

\keywords{Lagrangian concordance, Floer homology, Legendrian unknot,
non-symmetry}

\maketitle
\begin{abstract}
We derive constraints on Lagrangian concordances from Legendrian submanifolds of the standard contact sphere admitting exact Lagrangian fillings. More precisely, we show that such a
concordance induces an isomorphism on the level of bilinearised
Legendrian contact cohomology. This is used to prove the existence
of non-invertible exact Lagrangian concordances in all dimensions.
In addition, using a result of Eliashberg-Polterovich, we
completely classify exact Lagrangian concordances from the Legendrian unknot to itself in the tight contact-three sphere: every
such concordance is the trace of a Legendrian isotopy. We also discuss a high dimensional topological result related to this classification.
\end{abstract}

\markboth{Chantraine, Dimitroglou Rizell, Ghiggini, and
Golovko}{Floer homology and Lagrangian concordance}
\section{Introduction}
In this paper we are interested in exhibiting various rigidity phenomena for Lagrangian concordances between Legendrian submanifolds of the standard contact sphere $S^{2n+1}$ (or equivalently of the standard contact space $\R^{2n+1}$). Recall that the standard contact structure on $S^{2n+1} \subset \R^{2n+2}$ is given by $\xi_{st}:=\ker(\theta_0|_{TS^{2n+1}})$, where $\theta_0:=\frac{1}{2}\sum_{i=1}^{2n+2}(x_i dy_i -y_idx_i)$ is a one-form on $\R^{2n+2}$. A submanifold $\Lambda \subset (S^{2n+1},\xi_{st})$ is Legendrian if it is $n$-dimensional and tangent to $\xi_{st}$. A Lagrangian concordance is a special case of an exact Lagrangian cobordism $L \subset (\R^{2n+2} \setminus \{0\},d\theta_0)$ which is diffeomorphic to $\R \times \Lambda$; see Definition \ref{defn: lagrangian concordance}. In particular, Lagrangian concordant Legendrian submanifolds are diffeomorphic.

The classical (complete) obstruction for the existence of immersed Lagrangian cobordisms is the formal Lagrangian class \cite{PartialDiffRel,ClassificationLagrangeImmersions}. For example, when $n=1$, an \emph{immersed} Lagrangian concordance between two knots exists if and only if the two knots have the same rotation numbers. The classical obstruction to the
existence of an \emph{embedded} (not necessarily exact) oriented Lagrangian cobordism comes
from the \emph{Thurston-Bennequin invariant}, which is defined as
the linking number ${\tt
tb}(\Lambda):=\operatorname{lk}(\Lambda,\Lambda')$, where $\Lambda'$
is obtained by pushing $\Lambda$ slightly in the direction of the
Reeb vector field. If $L$ is a (not necessarily exact) oriented
Lagrangian cobordism from $\Lambda^-$ to $\Lambda^+$, the
corresponding Thurston-Bennequin invariants are related by
\begin{equation}
\label{eq:tbeq}
{\tt tb}(\Lambda^+)-{\tt tb}(\Lambda^-) = (-1)^{\frac{1}{2}(n^2-3n)}\chi(L,\Lambda_+)\end{equation}
as was shown in \cite{LagrConc} and \cite{NoteLagCob}.

In high dimension, recent results of Eliashberg-Murphy in \cite{LagCaps} implies that exact Lagrangian cobordisms satisfy an $h$-principle similar to the one in \cite{ClassificationLagrangeImmersions,PartialDiffRel} when the negative end $\Lambda^-$ is loose as defined in \cite{LooseLeg}. Thus, in order to expect rigidity phenomena, in the present paper we will study Lagrangian cobordisms whose
negative end $\Lambda^-$ admits an exact Lagrangian filling, i.e.~an
exact Lagrangian cobordism from $\emptyset$ to $\Lambda^-$.

In fact, it is well-known that a Legendrian submanifold of
$(S^{2n+1},\xi_{st})$ which admits an exact filling cannot be loose. In the forthcoming paper \cite{Cthulhu}, we will study exact Lagrangian cobordisms between Legendrian submanifolds admitting augmentations. This condition, in particular, implies that these Legendrian submanifolds are not loose.  The geometrical situation is however more involved due to the fact that Legendrian submanifolds admitting augmentations are not necessarily fillable.

Observe that there are very few known examples and constructions of cobordisms in the case when $\Lambda^-$ admits an exact Lagrangian filling. It seems like the only ones are so-called \emph{decomposable} cobordisms, i.e.~exact Lagrangian cobordisms built by concatenating cobordisms of the following two types:
\begin{itemize}
\item The trace of a Legendrian isotopy \cite{LagrConc}, \cite[Lemma 4.2.5]{findimlagint}.
\item The elementary Lagrangian handle attachment corresponding to an ambient surgery on the Legendrian submanifold \cite{LegKnotsLagCob}, \cite{LegAmbient}.
\end{itemize}

Our goal is to extract obstructions to the existence of a Lagrangian concordance from a Legendrian submanifold$\Lambda^-$ that admits an exact Lagrangian filling from its Legendrian contact homology \cite{IntroSFT}, \cite{DiffAlg}, \cite{LegKnotsLagCob}. One of the first results in this direction was in
\cite{LagrConcordNotASymmRel}, where it is shown that the Legendrian
representative $\Lambda_{m(9_{46})} \subset (S^3,\xi_{st})$ of the
knot $9_{46}$ described in Figure~\ref{fig:LCHlag} (with maximal
Thurston-Bennequin invariant) is not Lagrangian concordant to the
unknot $\Lambda_0$ with ${\tt tb}(\Lambda_0)=-1$. 
On the other hand $\Lambda_{m(9_{46})}$ is Lagrangian fillable by a disc which can be seen as the concatenation of the standard Lagrangian disc that fills $\Lambda_0$ and a Lagrangian concordance from $\Lambda_0$ to $\Lambda_{m(9_{46})}$. In other words, this result shows that the relation of being Lagrangian concordant is not symmetric, in particular there are Lagrangian concordances that cannot be inverted (unlike those arising from the trace of a Legendrian isotopy).

In dimension three, this result was later generalised in \cite{LegendrianMonopole} and \cite{ObstructionsLagConc}. The latter article is based upon the technique of rulings, which is a combinatorial invariant related to Legendrian contact homology which can be defined in the case $n=1$. It is shown there that a Lagrangian
concordance from $\Lambda^-$ to $\Lambda_0$ imposes restrictions on the possible rulings of $\Lambda^-$. Namely, such a Legendrian knot
cannot have two normal rulings.

Our main result is similar in spirit, but concerns the bilinearised Legendrian contact cohomology induced by a pair of augmentations. This invariant can be defined in all dimensions.

Our main rigidity result for Lagrangian concordances is
the following:
\begin{Thm}
\label{thm:main}
Suppose that $L$ is a Lagrangian concordance from $\Lambda^- $ to $\Lambda^+$ in $\R^{2n+1}$.
If $\varepsilon_0,\varepsilon_1$ are augmentations of the Chekanov-Eliashberg algebra of $\Lambda^-$ that are induced by exact Lagrangian fillings, it follows that the induced bilinearised map
\[ \Phi^{\varepsilon_0,\varepsilon_1}_L \co LCH_{\varepsilon_0,\varepsilon_1}^\bullet (\Lambda^-) \to LCH_{\varepsilon _0\circ \Phi,\varepsilon_1\circ\Phi}^\bullet(\Lambda^+)\]
in cohomology is an isomorphism.
\end{Thm}

\begin{Rem}
  \begin{enumerate}
  \item The previous theorem applies in the more general case when the ambient contact manifold is a contactisation of a Liouville domain as defined in Section \ref{sec:setup}.
\item As was observed in \cite{NoteLagCob}, it follows from the isomorphism stated in Theorem \ref{malgthinstregbam} that there is an isomorphism
\[LCH_\varepsilon^\bullet (\Lambda^-) \simeq LCH_{\varepsilon\circ\Phi_L}^\bullet (\Lambda^+)\]
of linearised cohomologies in the case when $\varepsilon$ is induced
by a filling. Namely, both the left and the right-hand side are
isomorphic to the singular homology of the filling (with a shift in
grading).
  \end{enumerate}
\end{Rem}
In general, we will use $\Lambda_0 \subset  (\R^{2n+1},\xi_{st}) \subset (S^{2n+1},\xi_{st})$ to denote the standard Legendrian $n$-dimensional sphere, which can be obtained as the intersection $S^{2n+1} \cap \{ \mathbf{y}=0\}$. The standard representative of this Legendrian submanifold admits a unique augmentation of its Legendrian contact homology algebra. The following
important corollary follows from the fact that
$LCH_{\varepsilon}^\bullet (\Lambda_0)$ is
one-dimensional and concentrated in degree $n$ (see \cite[Example
4.2]{NonIsoLeg}) for this augmentation.
\begin{Cor}
If the Legendrian submanifold $\Lambda^- \subset (S^{2n+1},\xi_{st})$ admits two fillings inducing augmentations $\varepsilon_i$, $i=0,1$, for which $LCH_{\varepsilon_0,\varepsilon_1}^\bullet (\Lambda^-)$ is not one-dimensional and concentrated in degree $n$, it follows that there is no exact Lagrangian concordance to from $\Lambda^-$ to $\Lambda_0$.
\end{Cor}
When $n=1$, using the fact that there is a correspondence between normal rulings and augmentations \cite{ChekElInv,AugRul}, this statement is a special case of the results in  \cite[Theorem 1.2]{ObstructionsLagConc}.

In the case when $n=1$ a result due to Eliashberg and Polterovich \cite{LocalLagrangian} shows that, up to
compactly supported Hamiltonian isotopy, the unknot $\Lambda_0$ has
a unique filling by a disc. Then a fillable Legendrian knot $\Lambda$ which is exact Lagrangian cobordant to the unknot is automatically
\emph{doubly slice}, as observed in \cite[Theorem
1.2]{LagrConcordNotASymmRel}. A smooth knot in $S^3$ is called slice
if it bounds an embedded disc in $D^4$, and it is called doubly
slice if can be obtained as the transverse intersection of an
unknotted embedding of $S^2 \hookrightarrow \R^4$ with the unit
sphere $S^3 \subset \R^4$. In Theorem \ref{thm:concisotopy} we
elaborate on this idea to give a complete classification of exact
Lagrangian concordances from $\Lambda_0$ to itself up to compactly
supported Hamiltonian isotopy.

A similar classification result in high dimensions is out of reach of current technology. However, there are strong topological constraints in the case when $\Lambda^+ \subset (\R^{2n+1},\xi_{st})$ has a single generic Reeb chord and admits an exact Lagrangian filling (e.g.~this is the case for the standard representative of $\Lambda_0$). Restrictions on the topology and smooth structure of such a Legendrian submanifold have previously been established in \cite{ExoticSpheres,EkholmSmith} by Ekholm-Smith. Also, in the case when $\Lambda=\Lambda_0$, a result due to Abouzaid, Fukaya-Seidel-Smith, Nadler, and Kragh (see Remark \ref{rem:bull}) implies that any exact Lagrangian  is contractible. We generalise this result in Theorem \ref{thm:fukayaseidelsmith}.

We end with the following corollary. Using the result of Eliashberg-Polterovich it was shown in \cite[Theorem 6.1]{LagrConcordNotASymmRel} that every exact Lagrangian cobordism from the standard representative of the one-dimensional Legendrian unknot $\Lambda_0$ to itself is a concordance which, moreover, induces the identity automorphism of its Chekanov-Eliashberg algebra. Combining Theorem \ref{thm:main}, Theorem \ref{thm:fukayaseidelsmith}, together with an algebraic consideration, we obtain an analogous result in high dimensions as well:

\begin{Cor}\label{sec:introduction-cor}
Suppose that $n \neq 3,4$ and let $\Lambda^+ \subset (\R^{2n+1},\xi_{st})$ be a Legendrian submanifold having a single generic Reeb chord (e.g.~$\Lambda^+=\Lambda_0$), or suppose that $n=4$ and that $\Lambda^+ \simeq S^4$. Any exact Lagrangian cobordism $L$ from a Legendrian submanifold $\Lambda^-$ to $\Lambda$, where $\Lambda^-$ moreover admits an exact Lagrangian filling, induces a unital DGA morphism
\[ \Phi_L \colon (\mathcal{A}(\Lambda^+),\partial_{\Lambda^+}) \to  (\mathcal{A}(\Lambda^-),\partial_{\Lambda^-}) \]
which is a monomorphism admitting a left-inverse. In the case when $\Lambda^-$ satisfies the same assumptions as $\Lambda^+$, $L$ is moreover a concordance and the above map is an isomorphism.
\end{Cor}

We now give an outline of the paper. In Section \ref{sec:background}, we review all preliminary definitions and results in order to prove the main results of the paper. In particular, we recall the definitions of a Lagrangian cobordisms and concordances in Section \ref{sec:setup}. We recall the front spinning construction in Section \ref{sec:front-spinn-constr}. In sections  \ref{sec:bilin-legendr-cont} to \ref{sec:relat-betw-wrapp} we recall the definition of bilinearised Legendrian contact homology and its relation with the Floer homology of  Lagrangian fillings. In sections \ref{sec:moduli} to \ref{sec:relat-counts-pseudo} we define the moduli spaces of pseudo-holomorphic curves. Notably in Section \ref{sec:trans} we use the results by Lazzarini in \cite{ExistenceInjective} to prove a useful transversality property for pseudo-holomorphic curves with one positive puncture. Theorem \ref{thm:main} is proved in Section \ref{sec:proof-theor-refthm:m} by a computation using wrapped Floer homology. Section \ref{sec:lagr-conc-legendr} is devoted to the study of Lagrangian concordances from $\Lambda_0$ to $\Lambda_0$ inside $\R^4\setminus \{0\}$, which we classify up to Hamiltonian isotopy. We also give a proof of Theorem \ref{thm:fukayaseidelsmith} using wrapped Floer homology with local coefficients. We conclude in Section \ref{sec:nonsymm} with non-symmetry results for high-dimensional Lagrangian concordances in the above rigid settings.

\section*{Acknowledgements}
\label{sec:aknowledgments}
The first and third authors would like to thank the organisers of the 21st G\"{o}kova Topology Conference, May 2014.

\section{Background}\label{sec:background}
\subsection{Geometric set-up}
\label{sec:setup}

A \emph{contact manifold} is a pair $(Y,\xi)$ consisting of an
odd-dimensional manifold $Y$ together with a smooth maximally
non-integrable field of tangent hyperplanes $\xi \subset TY$. Given
the choice of a contact one-form $\alpha$, i.e. a one-form such that
$\xi=\ker \alpha$, the associated \emph{Reeb vector field
$R_\alpha$} is uniquely determined by
\[ \iota_{R_\alpha}\alpha=1, \:\: \iota_{R_\alpha}d\alpha=0.\]
The \emph{symplectisation} of $(Y,\alpha)$ is the exact symplectic
manifold $(\R \times Y,d(e^t\alpha))$, where $t$ denotes the standard coordinate of the $\R$-factor. In general, an \emph{exact
symplectic manifold} is a pair $(P,d\theta)$ consisting of an even-dimensional
smooth manifold $P$ together with an exact non-degenerate two-form
$d\theta$.

An $n$-dimensional submanifold $\Lambda \subset (Y^{2n+1},\xi)$ is called \emph{Legendrian} if $T\Lambda \subset \xi$. A half-dimensional submanifold, or immersed submanifold, $L \subset (P,d\theta)$ is called \emph{exact Lagrangian} if the pull-back of $\theta$ to $L$ is exact. We are interested in the following relation between Legendrian submanifolds.

\begin{defn}\label{def:lagcob}
A properly embedded exact Lagrangian submanifold $L \subset \R \times Y$ of the symplectisation without boundary is called an \emph{exact Lagrangian cobordism from $\Lambda^-$ to $\Lambda^+$} if it is of the form
\[ L = (-\infty,-M] \times \Lambda^- \: \cup \: \overline{L} \: \cup  \: [M,+\infty) \times \Lambda^+\]
for some number $M>0$, where $\overline{L} \subset [-M,M] \times Y$ is compact with boundary $\partial\overline{L}=\overline{L}\cap [-M,M]\times Y \cong \Lambda^-\sqcup\Lambda^+$, and if $e^t\alpha$ has a primitive which is constant on $L \cap \{ t \le -M \}$ (we refer to \cite{LegKnotsLagCob} and \cite{Glob_ex_cob} for details concerning the last condition).
\end{defn}

Observe that it follows from the definition that $\Lambda^\pm \subset Y$ necessarily are Legendrian submanifolds.

The Legendrian submanifold $\Lambda^{+}$ is called the  {\em positive end of $L$}, while $\Lambda^{-}$ is called the {\em negative end of $L$}. In the special case when $\Lambda^-=\emptyset$ we say that $L$ is an \emph{exact Lagrangian filling} of $\Lambda^+$.

Lagrangian concordances are particular examples of exact Lagrangian cobordisms.
\begin{defn}\label{defn: lagrangian concordance}
A \emph{Lagrangian concordance} $L$ from a Legendrian submanifold
$\Lambda^-\subset (Y,\alpha)$ to a Legendrian submanifold
$\Lambda^+\subset (Y,\alpha)$ in a symplectisation $(\R \times
Y,d(e^t\alpha))$ is a Lagrangian cobordism from $\Lambda^-$ to
$\Lambda^+$ whose compact part $\overline{L}$ is diffeomorphic to
$[0,1] \times \Lambda^-$.
\end{defn}
Note that if there exists a Lagrangian concordance from $\Lambda^-$
to $\Lambda^+$ then, in particular, $\Lambda^-$ is diffeomorphic to
$\Lambda^+$. The exactness of a Lagrangian concordance is immediate as all the topology is concentrated in the cylindrical ends, where the form $e^t\alpha$ vanishes.

The \textit{Maslov number} $\mu_L$ of a Lagrangian cobordism $L$ is defined to be the generator of the image of the \emph{Maslov class} $\mu: H_2(\R\times Y,L)\rightarrow \mathbb{Z}$ (see e.g.~\cite{NonIsoLeg} for more details). Its relevance comes from the fact that gradings in Floer homology are defined modulo $\mu_L$. In the special case of a Lagrangian concordance, grading consideration are vastly simplified as the Maslov number of a concordance is twice the rotation number of the Legendrian (positive or negative) end.

The \emph{contactisation} of an exact symplectic manifold $(P,d\theta)$ is the contact manifold $(P \times \R,dz+\theta)$ together with a natural choice of contact form, where $z$ denotes a coordinate on the $\R$-factor. The Reeb vector field with respect to this contact form is given by $\partial_z$.  The natural projection $\Pi_{\operatorname{Lag}}:P \times \R\to P$ is called the {\em Lagrangian projection}. Given a closed
Legendrian submanifold $\Lambda\subset P\times \R$, it follows that
$\Pi_{\operatorname{Lag}}(\Lambda) \subset (P,d\theta)$ is an exact Lagrangian immersion whose double points correspond to integral curves of the Reeb
vector field $\partial_{z}$ having endpoints on $\Lambda$; such
integral curves are called the {\em Reeb chords on $\Lambda$}.
The set of all Reeb chords on $\Lambda$ will be denoted by $\mathcal Q(\Lambda)$.
We say that $\Lambda$ is {\em chord generic} if the double points of $\Pi_{\operatorname{Lag}}(\Lambda)$ are transverse, which in particular implies that $|\mathcal Q(\Lambda)|<\infty$.

We will here mainly be interested in the following two contact manifolds:
the standard $\R^{2n+1}$ and the standard $S^{2n+1}$.
First we define the standard $\R^{2n+1}$. The contactisation of the standard symplectic vector space $(\R^{2n},-d \theta_{\R^n})$, where $\theta_{\R^n}=\sum_{i=1}^n y_idx_i$ is the Liouville form, is the standard contact $(2n+1)$-space $(\R^{2n}\times \R,\xi_{st}:=\ker \alpha_0)$, $\alpha_0:=dz-\theta_{\R^{n}}$.  Note that Gromov's theorem \cite{Gromov} implies that a closed exact Lagrangian immersion in $\R^{2n}$ must have at least one double-point or, equivalently, that any closed Legendrian submanifold of the standard contact vector space must have a Reeb chord.

Then we define the standard $S^{2n+1}$. Considering the primitive $\theta_0:=\frac{1}{2}\sum_{i=1}^{n+1}(x_i dy_i -y_idx_i)$ of the standard symplectic form on $\R^{2n+2}$, the standard contact sphere is given by $(S^{2n+1},\xi_{st}:=\ker \alpha_{st})$, where $\alpha_{st}:=\theta_0|_{TS^{2n+1}}$ is induced by the standard embedding $S^{2n+1} \subset \R^{2n+2}$ as the unit sphere. Observe that the complement of a point of the standard contact sphere is contactomorphic to the standard contact vector space \cite[Proposition 2.1.8]{IntroContact}. The exact symplectic manifold $(\R^{2n+2} \setminus \{0\},d\theta_0)$ can be identified with the symplectisation of $(S^{2n+1},\alpha_{st})$.

\subsection{The front spinning construction}\label{sec:front-spinn-constr}
Given a Legendrian submanifold $\Lambda \subset (\R^{2n+1},\xi_{st})$, the so called front spinning construction produces a Legendrian embedding of $\Lambda \times S^1$ inside $(\R^{2(n+1)+1},\xi_{st})$, as described by Ekholm, Etnyre and Sullivan \cite{NonIsoLeg}. In \cite{NoteOnFrontSpinning} this construction was extended to the $S^m$-front spinning, which produces a Legendrian embedding of $\Lambda \times S^m$ inside $(\R^{2(n+m)+1},\xi_{st})$. It was also shown that this construction extends to exact Lagrangian cobordisms. Below we recall these constructions.

The embedding
\begin{gather*}
\R \times S^n \hookrightarrow \R^{n+1},\\
(t,\mathbf{p}) \mapsto e^t\mathbf{p},
\end{gather*}
induces an embedding
\[ \R^n \times S^m = \R^{n-1} \times \R \times S^m \hookrightarrow \R^{n+m}\]
which, in turn, has a canonical extension to an embedding
\[\R^{2n} \times T^*S^m = T^*\R^n \times T^*S^m \hookrightarrow T^*\R^{n+m}=\R^{2(n+m)}\]
preserving the Liouville forms. Using $0_{S^m} \subset T^*S^m$ to denote the zero-section, the fact that $\Lambda \times 0_{S^m}$ is a Legendrian submanifold of the contactisation of $\R^{2n} \times T^*S^m$ thus provides an embedding of $\Lambda \times S^m$ into the contactisation of $\R^{2(n+m)}$.
\begin{defn}
Suppose that we are given a Legendrian submanifold $\Lambda \subset (\R^{2n+1},\xi_{st})$. The above Legendrian embedding of $\Lambda \times S^m$ is called the \emph{$S^m$-spin of $\Lambda$} and will be denoted by $\Sigma_{S^m}\Lambda \subset (\R^{2(n+m)+1},\xi_{st})$.
\end{defn}

Observe that the symplectisation $(\R \times \R^{2n+1},d(e^t\alpha_0))$ of the standard contact vector space is symplectomorphic to $(\R^{2(n+1)},d\theta_{\R^{n+1}})$. For an exact Lagrangian cobordism $L$ from $\Lambda^-$ to $\Lambda^+$, the image of the exact Lagrangian submanifold $L \times 0_{S^m} \subset \R^{2(n+1)} \times T^*S^m$ under the above embedding can be seen as an exact symplectic cobordism from $\Sigma_{S^m}\Lambda^-$ to $\Sigma_{S^m}\Lambda^+$ inside the symplectisation of $(\R^{2(n+m)+1},\xi_{st})$.
\begin{defn}
Suppose that we are given an exact Lagrangian cobordism from $\Lambda^-$ to $\Lambda^+$ inside the symplectisation of $(\R^{2n+1},\xi_{st})$. The above exact Lagrangian cobordism from $\Sigma_{S^m}\Lambda^-$ to $\Sigma_{S^m}\Lambda^+$ diffeomorphic to $L \times S^m$ is called the \emph{$S^m$-spin of $L$} and will be denoted by $\Sigma_{S^m}L$.
\end{defn}
We refer to \cite{NoteOnFrontSpinning, OnHolRegFlexEndoc, EstimNumbrReebChordLinReprCharAlg} for more details and properties of this construction. Finally, observe that the $S^m$-front spinning construction can also be seen as a special case
of the Legendrian product construction of Lambert-Cole, see \cite{LegendrianProducts}.

\subsection{Bilinearised Legendrian contact cohomology}
\label{sec:bilin-legendr-cont}

Legendrian contact homology (LCH) is a modern Legendrian isotopy
invariant defined for Legendrian submanifolds of $(\R^3,\xi_{st})$
by Chekanov \cite{DiffAlg} in a combinatorial way,
and later generalised to the contactisation $P \times \R$ of a Liouville domain
$(P,d\theta)$ by Ekholm, Etnyre and Sullivan in \cite{ContHomP} using holomorphic curves. This invariant can
be seen as a part of the symplectic field theory  program,
which was proposed by Eliashberg, Givental and Hofer in
\cite{IntroSFT}.

Given a chord-generic Legendrian submanifold $\Lambda$, we associate
to it a unital non-commutative differential graded algebra $\mathcal
A(\Lambda)$ over $\mathbb Z_2$ freely generated by the set of Reeb
chords $\mathcal{Q}(\Lambda)$. This algebra is sometimes called the
Chekanov-Eliashberg algebra of $\Lambda$. The differential
$\partial_{\Lambda}$ is defined on the generators by the count
 \[ \partial_\Lambda(a)= \sum_{\mathbf{b}} \# (\mathcal{M}_{\R \times \Lambda}(a;\mathbf{b})/\R)\mathbf{b}\]
of pseudo-holomorphic discs with one positive asymptotics to the
Reeb chord $a$ and several negative asymptotics to the Reeb chords
$\mathbf{b}=(b_1,\ldots,b_k)$. Here the sum is taken over the one-dimensional components of the moduli spaces, where the $\R$-action is induced by translation, for some generic choice of cylindrical almost complex structure. See Sections \ref{sec:moduli}, \ref{sec:trans} below for the definitions of
these moduli spaces. The differential is then extended to all of
$\mathcal{A}(\Lambda)$ via the Leibniz rule and linearity.
For the details of this construction we refer to \cite{ContHomP}.

In order to extract finite-dimensional linear information out of
this DGA, Chekanov considered augmentations as bounding cochains for Legendrian contact homology.

\begin{defn}
Let $(\mathcal{A},\partial)$ be a  DGA over a unital commutative ring $\mathcal{R}$. An \textit{augmentation} of $\mathcal{A}$ is a unital DGA map
$$\varepsilon:(\mathcal{A},\partial)\rightarrow (\mathcal{R},0),$$
where all elements of $\mathcal{R}$ are concentrated in degree $0$.
\end{defn}

In other words, an augmentation $\varepsilon$ is a unital algebra map such
that
\begin{itemize}
\item $\varepsilon(a)=0$ if $|a|\not= 0$,
\item $\varepsilon\circ\partial =0$.
\end{itemize}


Given two augmentations of the Chekanov-Eliashberg algebra
$(\mathcal{A}(\Lambda),\partial)$, we can define the bilinearised
Legendrian contact cohomology complex, which is the
finite-dimensional $\Z_2$-vector space
$LCC^\bullet_{\varepsilon_0,\varepsilon_1}(\Lambda)$ with basis
$\mathcal{Q}(\Lambda)$, whose boundary is given by the count
\[ d_{\varepsilon_0,\varepsilon_1}(c)=\sum \#(\mathcal{M}_{\R\times\Lambda}(a;\mathbf{b}c\mathbf{d})/\R) \varepsilon_0(\mathbf{b})\varepsilon_1(\mathbf{d})a\]
of pseudo-holomorphic discs, where the sum is taken over the rigid
components (modulo translation) of the moduli spaces for some
generic choice of cylindrical almost complex structure. The homology
of this complex will be denoted by
$LCH^\bullet_{\varepsilon_0,\varepsilon_1}(\Lambda)$.

This cohomology theory is a generalisation of the original
Chekanov's linearised Legendrian contact cohomology, see
\cite{DiffAlg}. If $\varepsilon_0=\varepsilon_1$, then these two theories coincide.
The set of isomorphism classes of bilinearised Legendrian cohomologies is a Legendrian isotopy invariant \cite[Theorem 1.1]{Bilinearised}.

In the spirit of symplectic field theory, Ekholm has shown \cite{RationalSFT} that an exact Lagrangian cobordism $L$ from $\Lambda_-$ to $\Lambda_+$ together with the choice of a generic almost complex structure induces a DGA morphism
\[\Phi_L: (\mathcal A(\Lambda^+),\partial_{\Lambda^+})\to \mathcal (A(\Lambda^-),\partial_{\Lambda^-})\]
defined on generators by
\begin{equation}\label{eq: dga morphism}
 \Phi_L(a)= \sum \# \mathcal{M}_L(a;\mathbf{b})\mathbf{b},
\end{equation}
where the sum is taken over the rigid components of the moduli spaces for some generic choice of compatible almost complex structure. In particular, an exact Lagrangian filling induces an augmentation because the Chekanov-Eliashberg algebra associated to the empty set is the ground ring.
 Moreover, given two augmentations $\varepsilon_i$, $i=0,1$, of $(A(\Lambda^-),\partial_{\Lambda^-})$ there is an induced chain map
\[ \Phi^{\varepsilon_0,\varepsilon_1}_L \co LCC^\bullet_{\varepsilon_0,\varepsilon_1}(\Lambda^-) \to LCC^\bullet_{\varepsilon_0\circ \Phi_L,\varepsilon_1\circ \Phi_L}(\Lambda^+)\]
which is defined by
\[ \Phi^{\varepsilon_0,\varepsilon_1}_L(c)= \sum \# \mathcal{M}_L(a;\mathbf{b}c\mathbf{d}) \varepsilon_0(\mathbf{b})\varepsilon_1(\mathbf{d})a.\]

We now consider the case $\Lambda=\Lambda_0 \cup \Lambda_1$, and when $\varepsilon_i$ is an augmentation of $(\mathcal{A}(\Lambda_i),\partial_{\Lambda_i})$. It follows from the fact that $\partial_\Lambda$ counts pseudo-holomorphic \emph{discs} (which in particular are connected) that there is an induced augmentation $\varepsilon$ of $(\mathcal{A}(\Lambda),\partial_\Lambda)$ which on
the Reeb chord $c \in \mathcal{Q}(\Lambda_i)$ takes the value $\varepsilon_i(c)$ and  vanishes on the Reeb chords from $\Lambda_0$ to $\Lambda_1$
and from $\Lambda_1$ to $\Lambda_0$. For the same reason, the subset

\[  \mathcal{Q}(\Lambda_1,\Lambda_0) \subset \mathcal{Q}(\Lambda)\]
consisting of Reeb chords starting on $\Lambda_1$ and ending on $\Lambda_0$ spans a subcomplex of $LCC^\bullet_\epsilon(\Lambda)$. We will denote this subcomplex by
\[(LCC^\bullet_{\varepsilon_0,\varepsilon_1}(\Lambda_0,\Lambda_1),d_{\varepsilon_0,\varepsilon_1}).\]

\begin{Rem}\label{rem: augmentations are locally constant}
If $\Lambda_0,\Lambda_1 \subset P \times \R$ are sufficiently $C^1$-close, it follows from the invariance theorem in \cite{ContHomP} that
the canonical identification of the generators induces an isomorphism between the  Chekanov-Eliashberg algebras $({\mathcal A}(\Lambda_0), \partial_{\Lambda_0})$ and $({\mathcal A}(\Lambda_1), \partial_{\Lambda_1})$. In particular, there is a canonical bijective correspondence between the augmentations of $(\mathcal{A}(\Lambda_0),\partial_{\Lambda_0})$ and $(\mathcal{A}(\Lambda_1),\partial_{\Lambda_1})$ in this case.
\end{Rem}

Use $\phi^t \colon P \times \R \to P \times \R$ to denote the translation $z \mapsto z +t$, i.e.~the time-$t$ Reeb flow for the standard contact form.
\begin{Prop}
\label{prop:identification}
Consider the cylindrical lift $J$ of a fixed regular almost complex structure on $P$ (see Section \ref{sec:trans}). If  the Legendrian submanifold $\Lambda' \subset P \times \R$ is sufficiently $C^1$-close to $\phi^\epsilon(\Lambda) \subset P \times \R$, for each $0<\epsilon<\min_{c \in \mathcal{Q}(\Lambda)}\ell(c)$ there is a canonical isomorphism
\[(LCC^\bullet_{\varepsilon_0,\varepsilon_1}(\Lambda,\Lambda'),d_{\varepsilon_0,\varepsilon_1}) \simeq (LCC^\bullet_{\varepsilon_0,\varepsilon_1}(\Lambda),d_{\varepsilon_0,\varepsilon_1})\]
of complexes, where we have identified the augmentations of the Chekanov-Eliashberg algebras of $\Lambda$ and $\Lambda'$ by Remark \ref{rem: augmentations are locally constant}, and used $J$ in the definition of the differentials.
\end{Prop}
\begin{proof}
This follows by the statements in \cite[Section 6.1.2]{Lifting} which, in turn, follow from the analysis done in \cite{DualityLeg}.
\end{proof}

\subsection{Wrapped Floer homology}\label{wfh}
Wrapped Floer homology is a version of Lagrangian intersection Floer homology
for certain non-compact Lagrangian submanifolds. It first appeared in \cite{AbbonFloer}, and different versions were later developed in \cite{OpenStringAnalogue}, \cite{OnWrapped}, \cite{SympGeomCot}, and \cite{RationalSFT2}. We will be using the set-up of the latter version, which is useful for establishing a connection between wrapped Floer homology and bilinearised Legendrian contact cohomology.

In the following we will let $L_0,L_1 \subset \R \times P \times \R$ be exact Lagrangian fillings of $\Lambda_0,\Lambda_1 \subset P \times \R$, respectively, which are assumed to intersect transversely, and hence in a finite set of double-points. We use $\varepsilon_i$ to denote the augmentation induced by $L_i$, $i=0,1$. The wrapped Floer homology complex is defined to be
\[ CW_\bullet(L_0,L_1)=CW^\infty_\bullet(L_0,L_1) \oplus  CW^0_\bullet(L_0,L_1) , \]
where
\begin{eqnarray*}
CW^0(L_0,L_1) & := & \Z_2 \langle L_0 \cap L_1 \rangle \\
CW^\infty_\bullet(L_0,L_1) & := & LCC^{\bullet-1}_{\varepsilon_0,\varepsilon_1}(\Lambda_0,\Lambda_1).
\end{eqnarray*}
We refer to \cite{RationalSFT2} and \cite{Lifting} for details regarding the grading, which depends on the choice of a Maslov potential. The differential is defined to be of the form
\[ d=\begin{pmatrix}
d_\infty & \delta \\
0 & d_0
\end{pmatrix}
\]
with respect to the above decomposition, where
\begin{eqnarray*}
d_\infty & := & d_{\varepsilon_0,\varepsilon_1},\\
d_0(x) & := & \sum \#\mathcal{M}_{L_0,L_1}(y;x) y,\\
\delta(x) & := & \sum \#\mathcal{M}_{L_0,L_1}(a;x) a.
\end{eqnarray*}
Here the above sums are taken over the rigid components of the moduli spaces, and $x,y \in L_0 \cap L_1$, while $a \in \mathcal{Q}(\Lambda_1,\Lambda_0)$.

It immediately follows that $CW^\infty(L_0,L_1) \subset (CW(L_0,L_1),d)$ is a subcomplex, whose corresponding quotient complex can be identified with the complex $(CW^0(L_0,L_1),d_0)$. For obvious reasons, we will denote the latter quotient complex by
\[ CF_\bullet(L_0,L_1):=CW^0_\bullet(L_0,L_1),\]
since the differential of this complex counts ordinary Floer strips.

In the current setting, we have the following invariance result
\begin{Prop}[Proposition 5.12 in \cite{Lifting}]
\label{prop:invariance}
If $L_1,L_2 \subset \R \times P \times \R$ are exact Lagrangian fillings inside the symplectisation of a contactisation, it follows that $(CW_\bullet(L_0,L_1),d)$ is an acyclic complex or, equivalently, that
\[ \delta \colon CF_\bullet(L_0,L_1) \to LCC^\bullet_{\varepsilon_0,\varepsilon_1}(\Lambda_0,\Lambda_1)\]
is a quasi-isomorphism, where $\varepsilon_i$ denotes the augmentation induced by $L_i$, $i=0,1$.
\end{Prop}

\subsection{Relations between wrapped Floer homology and bilinearised LCH}
\label{sec:relat-betw-wrapp}
In the case when $L$ is an exact Lagrangian filling of $\Lambda$, the DGA morphism $\Phi_L$ described in Equation \ref{eq: dga morphism} is an augmentation of the Chekanov-Eliashberg algebra of $\Lambda$ defined by counting elements in the moduli space $\mathcal{M}_L(c)$ for each Reeb chord $c$ on $\Lambda$.

In~\cite{RationalSFT2}, Ekholm outlined an isomorphism, first conjectured by Seidel, relating the linearised Legendrian contact cohomology and the singular homology of a filling. The details of this isomorphism were later worked out in \cite{Lifting}.
\begin{Thm}\label{malgthinstregbam}
Let $\Lambda\subset P\times\R$ be a Legendrian submanifold admitting an exact Lagrangian filling $L$. There is an isomorphism
\begin{align*}
H_{i}(L; \Z_{2})\simeq LCH_{\varepsilon}^{n-i}(\Lambda),
\end{align*}
where $\varepsilon$ is the augmentation induced by $L$.
 Here all the gradings are taken modulo the Maslov number of $L$.
\end{Thm}
The above theorem follows from the following basic result, together with a standard computation. Recall that the Hamiltonian flow $\phi^\epsilon_{e^t}$ is simply a translation of the $z$-coordinate by $\epsilon$.
\begin{Thm}[Theorem 4.2 in \cite{Bilinearised}]
\label{thm:isom}
Let $\Lambda\subset P\times\R$ be a Legendrian submanifold admitting exact Lagrangian fillings $L_i$ inducing the augmentations $\varepsilon_i$ of $(\mathcal{A}(\Lambda),\partial_\Lambda)$, $i=0,1$. For each sufficiently small $\epsilon>0$ and an appropriate choice of compatible almost complex structure there is an isomorphism
\begin{align*}
HF_\bullet (L_0,\phi^\epsilon_{e^t}(L_1)) \simeq LCH_{\varepsilon_0,\varepsilon_1}^\bullet (\Lambda).
\end{align*}
\end{Thm}
\begin{proof}
The invariance in Proposition \ref{prop:invariance} shows that
\[HF_\bullet (L_0,\phi^\epsilon_{e^t}(L_1) \simeq LCH_{\varepsilon_0,\varepsilon_1}^\bullet (\Lambda,\phi^\epsilon(\Lambda)).\]
Choosing an almost complex structure appropriately, we can apply Theorem \ref{thm:pushoff} and Proposition \ref{prop:identification} to obtain the equality
\[ LCH_{\varepsilon_0,\varepsilon_1}^\bullet (\Lambda,\phi^\epsilon(\Lambda))=LCH_{\varepsilon_0,\varepsilon_1}^\bullet (\Lambda),\]
given that $\epsilon>0$ is sufficiently small.
\end{proof}

The bilinearised Legendrian contact homology of the $S^m$-spin $\Sigma_{S^m}\Lambda$ induced by a pair $\Sigma_{S^m}L_0$, $\Sigma_{S^m}L_1$ of $S^m$-spins of fillings can be computed using the following K\"unneth-type formula.
\begin{Thm}\label{kunnethformulaspinbilinearised}
Let $\varepsilon_i$ be the augmentation of
$(\mathcal{A}(\Lambda),\partial_\Lambda)$ induced by an exact Lagrangian filling $L_i$ and let
$\widetilde{\varepsilon_i}$ be the augmentation of
$(\mathcal{A}(\Sigma_{S^m}\Lambda),\partial_{\Sigma_{S^m}\Lambda})$
induced by the exact Lagrangian filling $\Sigma_{S^m}L_i$, $i=0,1$. There is an isomorphism of
graded $\Z_2$-vector spaces
\[(LCH^\bullet_{\widetilde{\varepsilon_0},\widetilde{\varepsilon_1}}(\Sigma_{S^m}\Lambda)) \simeq (LCH^\bullet_{\varepsilon_0,\varepsilon_1}(\Lambda))\otimes (H_\bullet(S^m;\Z_2)). \]
\end{Thm}
\begin{proof}
Use the K\"unneth-type formula for Lagrangian Floer homology, see e.g.~\cite{LiKunneth} or \cite[Section 2.6]{ReductionSymplectic}, together with the isomorphism in Theorem \ref{thm:isom}. Namely, $\Sigma_{S^m}L_i \subset (\R^{2(n+m+1)}=\R^{2(n+1)}\times\R^{2m},d\theta_0)$ has a neighbourhood symplectomorphic to $\R^{2(n+1)} \times \mathcal{N}_{0_{S^m}}T^*S^m$, where $\mathcal{N}_{0_{S^m}}T^*S^m$ is a neighbourhood of the zero section $0_{S^m}$ and $\Sigma_{S^m}L_i$ moreover is identified with $L_i \times 0_{S^m}$, $i=0,1$. The K\"{u}nneth-type formula now gives
$$(HF_\bullet(\Sigma_{S^m}L_0,\Sigma_{S^m}L_1)) \simeq (HF_\bullet(L_0,L_1)) \otimes (HF_\bullet(0_{S^m},0_{S^m})),$$
where the latter factor is isomorphic to $H_\bullet(S^m;\Z_2)$ by a standard computation.
\end{proof}

Finally, observe that this K\"unneth-type formula is analogous to
the version for generating family homology proved in
\cite[Proposition
5.4]{FamiliesLegendrianSubmanifoldsGeneratingFamilies} for spins of
Legendrian manifolds admitting generating families.

\subsection{Pseudo-holomorphic discs with boundary on a Lagrangian cobordism}
\label{sec:moduli}

In this section we describe the moduli spaces of pseudo-holomorphic discs involved it the construction of Legendrian contact homology. Recall that a compatible almost complex structure $J$ on the symplectisation $(\R \times Y,d(e^t\alpha))$ of $(Y,\alpha)$ is \emph{cylindrical} if
\begin{itemize}
\item $J$ is invariant under translations of the $t$-coordinate;
\item $J\partial_t=R_\alpha$; and
\item $J(\ker(\alpha))=\ker(\alpha)$.
\end{itemize}
In the following we let $L \subset \R \times Y$ be an exact Lagrangian cobordism from $\Lambda^-$ to $\Lambda^+$ inside the symplectisation of $(Y,\alpha)$, where $L$ is assumed to be cylindrical outside of $[-M,M] \times Y$. Also, we let $J$ be a compatible almost complex structure on $\R \times Y$ which is cylindrical outside of a compact subset of $[-M,M] \times Y$.

The so-called \emph{Hofer-Energy} of a map $u \colon (\Sigma,\partial \Sigma) \to (\R \times Y,L)$ from a Riemann surface with boundary is defined as
\[ E_H(u) := \sup_{\varphi \in \mathcal{C}} \int_\Sigma u^*d(\varphi(t) \alpha),\]
where $\mathcal{C}$ is the set consisting of smooth functions $\varphi \colon \R \to [0,2e^M]$ satisfying $\varphi(t)=e^t$ for $t \in [-M,M]$, and $\varphi'(t)\ge 0$. Observe that the Hofer-Energy is non-negative whenever $u$ is $J$-holomorphic, i.e.~satisfies $J \circ du=du \circ i$, for an almost complex structure $J$ of the above form.

Consider the piecewise smooth function $\overline{\varphi} \colon \R \to [0,e^M]$ which satisfies $\overline{\varphi}(t)=e^{-M}$ for $t \le -M$,  $\overline{\varphi}(t)=e^t$ for $t \in [-M,M]$, while  ${\overline\varphi}(t)=e^M$ for $t \ge M$. We define the \emph{$d\alpha$-energy} of $u$ by
\[E(u):=\int_{\Sigma} u^* d(\overline{\varphi} \alpha).\]

We will study $J$-holomorphic discs $u\colon (D^2,\partial D^2) \to (\R \times Y, L)$, where the map $u$ is defined outside of a finite set of boundary points, usually called the \emph{(boundary) punctures}, and required to have finite Hofer energy.
At a puncture we require that either $t \circ u \to +\infty$, in which case we call the puncture \emph{positive}, or $t \circ u \to -\infty$, in which case we call the puncture \emph{negative}. The finiteness of the Hofer energy implies that $u$ is asymptotic to cylinders over Reeb chords at its boundary punctures. This fact follows by the same arguments as the analogous statement in the case when the boundary is empty and all punctures are internal, which was proven in \cite{PropPseudI}.

Let $a$ a Reeb chord on $\Lambda^+$ and $\mathbf{b}=b_1\cdot \hdots \cdot b_m$ a word of Reeb chords on $\Lambda^-$. We use
\[\mathcal{M}^J_L(a;\mathbf{b})\]
to denote the moduli space consisting of $J$-holomorphic discs $u$ as above that moreover satisfy the properties that:
\begin{itemize}
\item $u$ has a unique positive puncture $p_0 \in \partial D^2$, at which it is asymptotic to a cylinder over $a$; and
\item $u$ has $m$ negative punctures, and at the $i$-th negative puncture on the oriented boundary arc $\partial D^2 \setminus \{p_0\}$ it is asymptotic to a cylinder over $b_i$.
\end{itemize}

 Choose a primitive $f \colon L \to \R$ of $\overline{\varphi}\alpha |_{TL}$ which is constant on the negative end. Such a primitive exists by the definition, since $L$ is an exact Lagrangian cobordism of the form $e^t \alpha$, and $\overline{\varphi}\alpha|_{TL} = e^t\alpha|_{TL}$ because $\alpha$ vanishes on $L \cap \{|t|\ge M\}$,where
$L$ is cylindrical. For a Reeb chord $c$, we define
\[\ell(c):= \int_c \alpha >0.\]
For $a\in \mathcal{Q}(\Lambda^+)$ and $b \in \mathcal{Q}(\Lambda^-)$, we now write
\begin{eqnarray*}
\mathfrak{a}(b) & := &e^{-M}\ell(b),\\
\mathfrak{a}(a) & := & e^M\ell(a)+f(a_s)-f(a_e),
\end{eqnarray*}
where $a_s,a_e \in \Lambda^+$ denote the starting and the ending points of the Reeb chord $a$, respectively.

A standard computation utilising Stoke's theorem and the assumption that $L$ is exact shows that
\begin{Prop}
The $d\alpha$-energy of $u \in \mathcal{M}^J_L(a;\mathbf{b})$, where $J$ satisfies the above properties, is given by
\begin{equation}
\label{eq:energy}
0 \le E(u)=\mathfrak{a}(a)-(\mathfrak{a}(b_1) + \hdots + \mathfrak{a}(b_m)) \le E_H(u)
\end{equation}
where, in the case $M>0$, $E(u)=0$ if and only if $u$ is constant. In particular, $u$ must have a positive puncture unless it is constant.
\end{Prop}

We will also be interested in the case when $L=L_0 \cup L_1$, where each component $L_i$ is embedded, but where we allow intersections $L_0 \cap L_1 \neq \emptyset$ inside $[-M,M] \times Y$. Suppose that $a$ and $b$ both start on $L_1$ and end on $L_0$, $c_i$ has both endpoints on $L_0$, and $d_i$ has both endpoints on $L_1$. Letting $\mathbf{c}=c_1 \cdot \hdots \cdot c_{m_0}$ and $\mathbf{d}=d_1 \cdot \hdots \cdot c_{m_1}$, we write
\[\mathcal{M}^J_{L_0,L_1}(a;\mathbf{c},b,\mathbf{d}) := \mathcal{M}^J_{L_0\cup L_1}(a;\mathbf{c}b\mathbf{d}).\]
The reason for this notation is that, using an appropriate conformal identification of the domain, we will consider such a disc to be a $J$-holomorphic strip
\[ u \colon (\R \times [0,1],\R \times \{0\},\R \times \{1\}) \to (\R \times Y, L_0, L_1) \]
having one boundary component on $L_0$ and one boundary component on $L_1$ (although possibly having additional negative punctures on each boundary arc). Here $u$ is asymptotic to a cylinder over $a$ and $b$ as $s \to -\infty$ and $+\infty$, respectively, using the coordinates $(s,t)$ on $\R \times [0,1] =\{s+it; \:t \in [0,1]\} \subset \C$.

We will also consider the moduli spaces $\mathcal{M}^J_{L_0,L_1}(a;\mathbf{c},b,\mathbf{d})$, where we allow any of $a$ or $b$ to be double-points in $L_0 \cap L_1$, in which case a strip $u \in \mathcal{M}^J_{L_0,L_1}(a;\mathbf{c},b,\mathbf{d})$ is required to converge $a$ and $b$ as $s \to -\infty$ and $+\infty$, respectively. For a double-point $p \in L_0 \cap L_1$, we define
\[\mathfrak{a}(p):=f_1(p)-f_0(p)\]
where $f_i \colon L_i \to \R$ are potentials of  $\overline{\varphi}(t)\alpha$ which are required to coincide on the negative ends. Similarly to Formula \eqref{eq:energy} one can compute
\begin{Prop}
Let $u \in \mathcal{M}^J_{L_0,L_1}(p;\mathbf{c},q,\mathbf{d})$, where $p$ and $q$ are allowed to be either Reeb chords or double points, and suppose that $J$ is an almost complex structure satisfying the above properties. It follows that
\begin{equation}
\label{eq:energy2}
0 \le E(u)=\mathfrak{a}(p)-\mathfrak{a}(q)- \left( \sum_{i=1}^{m_0} \mathfrak{a}(c_i) + \sum_{i=1}^{m_1} \mathfrak{a}(d_i)\right) \le E_H(u)
\end{equation}
where, in the case $M>0$, $E(u)=0$ if and only if $u$ is constant.
\end{Prop}

\subsection{Transversality results}
\label{sec:trans}
In the case when $L=\R \times \Lambda$ is a trivial cylinder and $Y=P \times \R$ is a contactisation, we will chose $J=\widetilde{J}_P$ to be the uniquely defined \emph{cylindrical lift} of a compatible almost complex structure $J_P$ on $P$, i.e.~the cylindrical almost complex structure for which the canonical projection $\R \times P \times \R \to P$ is $(\widetilde{J}_P,J_P)$-holomorphic. Observe that the moduli spaces of $J_P$-holomorphic discs with boundary on the exact Lagrangian immersion $\Pi_{\operatorname{Lag}}(\Lambda)\subset P$ having one positive puncture is transversely cut-out for a suitable generic choice of $J_P$ by \cite{ContHomP}. Finally, the latter moduli spaces being transversely cut out implies that the moduli spaces $\mathcal{M}^{\widetilde{J}_P}_{\R \times \Lambda}(a;\mathbf{b})$ are transversely cut out as well \cite[Theorem 2.1]{Lifting}.

In the setting when $L$ is not cylindrical, the following technical result will be crucial for achieving transversality. Recall that that a pseudo-holomorphic map $u\colon (\Sigma,\partial \Sigma) \to (X,L)$ from a punctured Riemann surface is called \emph{simple} if the subset
\[\{p \in \Sigma; \:\: d_p u \neq 0, \:u^{-1}(u(p))=\{p\}\} \subset \Sigma \]
is open and dense. Standard techniques \cite{JholCurves} show that the simple pseudo-holomorphic curves are transversely cut out solutions for a generic almost complex structure.

We will prove that a disc with boundary on an exact Lagrangian cobordism in $\R \times P \times \R$ having exactly one positive puncture asymptotic to a cylinder over the Reeb chord $a$ is simple. Otherwise the techniques in \cite{ExistenceInjective}, \cite{RelativeFrames} could be used to extract a non-constant pseudo-holomorphic disc with boundary on $L$ without any positive puncture, thus leading to a contradiction. Apart from exactness, the following property is crucial here: the Reeb chord $a$ is an embedded integral curve and hence that $u$ is an embedding onto its image inside the subset $\{ t \ge N\}$ for $N>0$ sufficiently large.
\begin{Thm}
Let $L \subset \R \times P \times \R$ be an exact Lagrangian cobordism. Then $u \in \mathcal{M}^J_L(a;\mathbf{b})$ is simple.
\end{Thm}
\begin{proof}
We start by observing that $u$ is an embedding near each of its punctures because of the asymptotic properties. Moreover, for sufficiently large $N \gg 0$ we may assume that $u|_{u^{-1}\{ t \ge N\}}$ is arbitrarily close to a parametrisation of $[N, +\infty) \times a$ in the $C^1$-topology while $u|_{u^{-1}\{ t \le -N\}}$ is arbitrarily close to a parametrisation of $\bigcup_{i=1}^m (-\infty,-N] \times b_i$ in the $C^1$-topology.

Let $p_0 \in \partial D^2$ denote the positive puncture, and $p_i \in \partial D^2$, $i=1,\hdots,m$, the negative punctures. 
We define
\[U_0:=u^{-1}\{ t \ge N +1\} \subset D^2,\]
and
\[\bigcup_{i=1}^m U_i=u^{-1}\{ t \le -N -1\} \subset D^2,\]
where each $U_i \subset D^2$ is a connected punctured neighbourhood of $p_i$ satisfying $U_i \cap U_j = \emptyset$ for $i\neq j$. As a consequence of Carleman's similarity principle, see e.g.~\cite[Lemma 4.2]{ExistenceInjective}, $u|_{U_i}$ and $u|_{U_j}$ either intersect in a discrete set, or coincide after a holomorphic identification $U_i \simeq U_j$ of the domains. In particular, the domain $D^2 \setminus (U_0 \cup U_1 \cup \hdots \cup U_m)$ is contained in a closed domain $V \subset D^2 \setminus \{ p_0,\hdots,p_m\}$ diffeomorphic to a disc with smooth boundary and moreover satisfying the following property: the restriction of $u$ to two arcs in $\partial V \setminus \partial D^2$ are embeddings which either are disjoint or coincide.

For a sufficiently small neighbourhood $X \subset \R \times P \times \R $ of $u(V)$, we may use the above property of $u|_V$ to find a closed embedding of a Lagrangian submanifold $\widetilde{L} \subset X$ (with no boundary) for which $u(\partial V) \subset \widetilde{L}$. To that end, observe that $u(\partial V \setminus \partial D^2)$ is embedded and isotropic, while $u(V \cap \partial D^2) \subset L \cap X$ already is contained in a Lagrangian submanifold. The standard neighbourhood theorem for isotropic submanifolds can now readily be used to extend the isotropic curves $u(\partial V) \cap \{|t| \ge N  +1 \}$ to a Lagrangian submanifold coinciding with $L$ in a neighbourhood of $\{|t|=N+1 \}$.

For a Reeb chord $c \in \mathcal{Q}(\Lambda)$ we let $z_c \in \R$ be the midpoint of its image under the canonical projection $P \times \R \to \R$. The algebraic intersection number with the symplectic hypersurfaces
\begin{eqnarray*}
Y_0 & := & (\{N\} \times P \times \{z_a\}) \cap X,\\
Y_i & := & (\{-N\} \times P \times \{z_{b_i}\} \cap X, \:\: i=1,\hdots,m,
\end{eqnarray*}
which each may be assumed to intersect $u$ transversely in a single geometric point, induce cohomology classes $\eta_i \colon H_2(X,\widetilde{L}) \to \Z$ after possibly shrinking the neighbourhood $X \supset u(V)$ (thus making $Y_i$ disjoint from $X \cap \widetilde{L}$). An explicit calculation, together with an appropriate choice of orientation of the above symplectic hypersurfaces, gives

\[\eta_i(u|_V)=\begin{cases}
1, & i=0,\\
| \{ j \in \{1,\hdots,m\}; b_j=b_i \}| \ge 1, & i=1,\hdots,m.
\end{cases}\]
It is moreover the case that
\begin{equation}
\label{eq:positivity} \eta_i(v) \ge 0, \:\: i=0,\hdots,m,
\end{equation}
for any pseudo-holomorphic curve $v$ having boundary on $\widetilde{L}$ and image contained in $u(V)$, where equality holds if and only if $v$ is disjoint from $Y_i$. This fact follows from the fact that this is the case for $u$ (since it's image intersects $Y_i$ transversely in a single geometric point) and that both $u$ and $v$ can be thought of as holomorphic maps from $V$ to $(u(V),J|_{Tu(V)})$, both which are one-dimensional complex domains. The open mapping theorem thus implies that each geometric intersection of $v$ and $Y_i$ contributes positively.

We now compute the symplectic area $\int v^*d(e^t\alpha)$ where $v$ is a non-constant pseudo-holomorphic curve having boundary on $\widetilde{L}$ and image contained in $u(V)$. This computation is similar to the computation of \eqref{eq:energy} using Stoke's theorem with the only difference that, while $\widetilde{L} \cap \{ |t| \le N +1 \}$ is exact, $\widetilde{L}$ is not. By an application of the open mapping theorem in one-dimensional complex analysis, using the fact that $u|_V$ is embedded near $Y_i$, we can compute the behaviour of the boundary of $v$ in the region $\{ |t| \ge N +1 \}$ in terms of $\eta_i(v)$. It follows that there are constants $C_i>0$, $i=0,\hdots,m$, for which
\begin{equation}
\label{eq:areaest}
0 \le \int_v d(e^t\alpha) = C_0\eta_0(v)-(C_1\eta_1(v) + \hdots + C_m\eta_m(v)), \:\: [v] \in H_2(X,\widetilde{L}),
\end{equation}
where equality holds if and only if $v$ is constant.

We note that \cite[Theorem A]{RelativeFrames} can be applied to the $J$-holomorphic disc $u|_V$ having compact image in $X$ and boundary on $\widetilde{L}$. In particular, there are $J$-holomorphic discs $v_1,\hdots,v_l$ in $X$ with boundary on $\widetilde{L}$, where the $v_i$ are simple and satisfy $v_i(D^2) \subset u(V)$ together with
\begin{equation}
\label{eq:lazzarini}
[u|_V]=k_1[v_1]+\hdots+k_l[v_l] \in H_2(X,\widetilde{L}), \:\: k_i > 0.
\end{equation}
The refined statement \cite[Proposition 5.4]{RelativeFrames} moreover gives Riemann surfaces $\Sigma_i$ with boundary and $k_i$-fold coverings $h_i \colon (\Sigma_i,\partial \Sigma_i) \to (D^2,\partial D^2)$ together with embeddings $\iota_i \colon \Sigma_i \to D^2$, for which
\[ v_i \circ h_i=u|_V \circ \iota_i\]
is satisfied.

Formulas \eqref{eq:areaest} and \eqref{eq:positivity} imply that each $v_i$ must satisfy $\eta_0 > 0$ and hence, by Formula \eqref{eq:lazzarini}, we must have $l=1$ and $k_1=1$. An area consideration moreover shows that the image $\iota_1(\Sigma_1 \setminus \partial \Sigma_1) \subset D^2$ must be open and dense. It now follows that $u|_V$, and hence $u$, is simple.
\end{proof}

\subsection{Gromov-Hofer compactness under a neck-stretching sequence}
\label{sec:neckstretch}
We first need to recall some facts about the compactness properties of pseudo-holomorphic curves under a \emph{neck stretching sequence} as described in \cite[Section 3.4]{CompSFT}, \cite[Section 1.3]{IntroSFT} and \cite{CompAbbas}. We will restrict our attention to stretching the neck along a hypersurface in $(\R \times Y,d(e^t\alpha))$ of contact type being either of the form $\{ t_0 \} \times Y$, or $\{ t_0, t_1\} \times Y$ where $t_0<t_1$, and when $L \subset \R \times Y$ is an exact Lagrangian cobordism from $\Lambda_-$ to $\Lambda_+$. In some neighbourhood of $\{ t_i \} \times Y$ we make the assumption that $J$ coincides with the cylindrical almost complex structure $J_i$ and that $L$ coincides with a cylinder over $\Lambda_i \subset Y$. Moreover, we let $J_-$ and $J_+$ denote the cylindrical almost complex structure coinciding with $J$ on the subset $\{ t \le -M \}$ and $\{ t \ge M \}$, respectively.

In this case, we say that a sequence of almost complex structures $J^\tau$ stretches the neck around the above hypersurface if $J^\tau$ is equal to $J$ except in a neighbourhood of the hypersurface, where it is determined by
\begin{itemize}
\item $J^\tau|_{\ker \alpha}=J$; and
\item $J^\tau \partial_t = (1+\tau \sigma(t))R_\alpha$,
\end{itemize}
where $\sigma \colon \R \to \R_{\ge 0}$ is a smooth function satisfying $\sigma(t_i)=1$ and whose support is contained in some sufficiently small neighbourhood of the hypersurface.

In the first case we let $L_\alpha$ and $L_\beta$ be the cylindrical completions inside $\R \times Y$ of $L \cap \{ t \le t_0\}$ and $L\cap \{ t \ge t_0 \}$ respectively, and similarly we will let $J_\alpha$ and $J_\beta$ be the corresponding cylindrical completions of the restrictions of $J$ to $\{ t \le t_0\}$ and $\{ t \ge t_0\}$, respectively.

In the latter case we let $L_\alpha$, $L_\beta$, and $L_\gamma$ be the cylindrical completions of $L \cap \{ t \le t_0\}$, $L \cap \{ t_0 \le t \le t_1 \}$, and $L \cap \{ t \ge t_1\}$, respectively, and $J_\alpha$, $J_\beta$, and $J_\gamma$ the cylindrical completions of $J$ restricted to $\{ t \le t_0\}$, $\{ t_0 \le t \le t_1 \}$, and $\{ t \ge t_1\}$, respectively.

Recall that the moduli space of pseudo-holomorphic curves in a symplectisation can be compactified into the space of so-called connected pseudo-holomorphic buildings. Roughly speaking, a pseudo-holomorphic building in $\R \times Y$ consists of a finite number of levels $1,2,3,\hdots$, where each level consists of a finite number of pseudo-holomorphic curves in $\R \times Y$ (possibly with boundary on a Lagrangian cobordism). The asymptotics of the curves on two consecutive levels, as well as the boundary conditions, are moreover required to match up, so that the natural compactification produces a connected piecewise smooth symplectic curve in $\R \times Y$. We refer to \cite{IntroSFT} and \cite{CompAbbas} for more details.

For the first case, the relevant pseudo-holomorphic buildings are those consisting of the following consecutive levels
\begin{itemize}
\item a non-trivial $J_-$-holomorphic building with boundary on $\R \times \Lambda_-$;
\item a level consisting of $J_\alpha$-holomorphic curves with boundary on $L_\alpha$;
\item a non-trivial $J_0$-holomorphic building with boundary on $\R \times \Lambda_0$;
\item a level consisting of $J_\beta$-holomorphic curves with boundary on $L_\beta$; and
\item a non-trivial $J_+$-holomorphic building with boundary on $\R \times \Lambda_+$,
\end{itemize}

The definition is analogous in the second case.

By the $\alpha$, $\beta$, and $\gamma$-levels of a pseudo-holomorphic building, we will mean the $J_\alpha$, $J_\beta$ and $J_\gamma$-holomorphic sub-buildings with boundary on $L_\alpha$, $L_\beta$, and $L_\gamma$, respectively, appearing in the building.

The Gromov-Hofer compactness theorem in this setting states that a sequence of $J^\tau$-holomorphic curves with boundary on $L$, $\tau \to +\infty$, has a subsequence which converges to a pseudo-holomorphic building of the above form.

Conversely, under the further assumption that every building as above consists of transversely cut-out components, pseudo-holomorphic gluing gives a bijection transversely cut-out rigid $J^\tau$-holomorphic curves with boundary on $L$ for $\tau \gg 0$ sufficiently large, and pseudo-holomorphic buildings as above in which every component is transversely cut-out and rigid.

\subsection{Relating the counts of pseudo-holomorphic discs on a cobordism and its two-copy}\label{sec:relat-counts-pseudo}
It will be crucial to control the behaviour of certain pseudo-holomorphic discs on the two-copy $L \cup L'$, where $L$ is an exact Lagrangian cobordism and $L'$ is arbitrarily $C^1$-close to $L$. To that end, we will establish a substitute for a special case of the conjectural analytic result \cite[Lemma 4.11]{RationalSFT2}, which would give a bijection between the pseudo-holomorphic discs with boundary on $L \cup L'$ and pseudo-holomorphic discs with boundary on $L$ together with certain gradient flow lines.

Choose a smooth cut-off function $\rho \colon \R \to [0,1]$ satisfying $\rho(t)=0$ for $t \le 1$, $\rho'(t) \ge 0$ for all $t$, and $\rho(t)=1$ for $t \ge 2$. For each $N \ge 0$, we use $\rho$ to construct the smooth cut-off function $\rho_N \colon \R \to [0,1]$ which is determined by the property $\rho_N(-t)=\rho_N(t)$, together with $\rho_N(t)=\rho(t-(M+N))$ for $t \ge 0$.

Consider the autonomous Hamiltonian $h_N(t,p,z):=e^t \rho_N(t)$ on $\R \times P \times \R$. Observe that $\phi^\epsilon_{h_N}(L)$ is an exact Lagrangian cobordism from $\Lambda^-_\epsilon$ to  $\Lambda^+_\epsilon$, where the latter are Legendrian submanifolds obtained from $\Lambda^-$ and $\Lambda^+$ by the time-$\epsilon$ map of the Reeb flow. We will take $\epsilon>0$ sufficiently small so that, in particular, $\epsilon < \min_{c \in \mathcal{Q}(\Lambda^+)\cup\mathcal{Q}(\Lambda^-)}\ell(c)$. Similarly, we will also consider the autonomous Hamiltonians $h_\pm$ on $\R \times P \times \R$ given by $e^t \rho(\pm t)$, together with the time-$\epsilon$ maps $\phi^\epsilon_{h_\pm}$ of the corresponding flows.

Let $a$ be a Reeb chord on $\Lambda^+$, while $\mathbf{c}$ and $\mathbf{d}$ are words of Reeb chords on $\Lambda^-$, and $b$ is a Reeb chord on $\Lambda^-$. We will let $a^\epsilon$ denote the unique Reeb chord from $\Lambda^+_\epsilon$ to $\Lambda^+$ which is contained inside $a$ and, similarly, we will let $b^\epsilon$ denote the unique Reeb chord from $\Lambda^-_\epsilon$ to $\Lambda^-$ which is contained inside $b$. We will also write $\mathbf{d}^\epsilon=d_1^\epsilon \cdot \hdots \cdot d_{m_1}^\epsilon$, where $d_i^\epsilon$ is the Reeb chord on $\Lambda^-_\epsilon$ obtained as the image of $d_i$ under the time-$\epsilon$ map of the Reeb flow.

\begin{Thm}
\label{thm:pushoff}
Let the compatible almost complex structure $J$ on $\R \times P \times \R$ be regular, where $J$ coincides with the cylindrical lift $\widetilde{J}_P$ of the regular almost complex structure $J_P$ on $P$ outside of $[-M,M] \times P \times \R$. If $L'$ is an exact Lagrangian cobordism that is $C^1$-close to $L$ and coincides with $L$ outside of some fixed compact set, $\epsilon>0$ is sufficiently small and $N>0$ sufficiently large, then there are bijections
\begin{eqnarray}
 \mathcal{M}^J_{\phi^\epsilon_{e^t}(L')}(d^\epsilon;\mathbf{d}^\epsilon) & \simeq & \mathcal{M}^J_L(d;\mathbf{d}) \\ \label{correspondence}
\mathcal{M}^J_{L,\phi^\epsilon_{h_N}(L')}(a^\epsilon;\mathbf{c},b^\epsilon,\mathbf{d}^\epsilon) & \simeq & \mathcal{M}^J_L(a;\mathbf{c}b\mathbf{d}),
\end{eqnarray}
of rigid moduli spaces. Moreover, if $L$ is spin, these bijections preserve the coherent orientations.
\end{Thm}
\begin{proof}
It suffices to prove the claim for $L'=L$. In fact, by standard compactness and index arguments, a sufficiently small perturbation of $L'$ in the $C^1$-topology induces a bijection between the rigid moduli spaces for a regular almost complex structure $J$.

Since $(\phi^\epsilon_{e^t})^{-1}$ is a biholomorphism outside of $[-M,M] \times P \times \R$ and maps $\phi^{\epsilon}(L)$ to $L$, the first bijection follows for $\epsilon>0$ sufficiently small. Namely, the solutions on the left-hand side can be seen as solutions to a compactly supported perturbation of the boundary-value problem defining the moduli space on the right-hand side.

The second bijection requires more work. We stretch the neck around the hypersurfaces $\{t=\pm( M+1)\} \subset \R \times P \times \R$ of contact type, as described in Section \ref{sec:neckstretch}, by considering the boundary condition $L \cup \phi^\epsilon_{h_N}(L')$ for $N=1$ together with the sequence of almost complex structures $J^\tau$.  In the current setting, for each $\tau \ge 0$, there is a biholomorphism of the form
\begin{gather*}
(\R \times Y,J) \to (\R \times Y,J^\tau),\\
(t,y) \mapsto (f(t),y).
\end{gather*}
In other words, $J^\tau$-holomorphic curves with boundary on $L \cup \phi^\epsilon_{h_1}(L')$ are in bijective correspondence with the $J$-holomorphic curves with boundary on $L \cup \phi^\epsilon_{h_{N(\tau)}} (L)$ for some $N(\tau)$, where $N(\tau)$ can be seen to be monotone increasing with $\tau$.

We now proceed to investigate the possible holomorphic buildings appearing as the limit of the moduli spaces on the left-hand side under a neck-stretching sequence as above. To that end, we have used Lemma \ref{lem:breaking} in order to justify that the SFT compactness theorem can be applied.

First, the $\beta$-level of such a building consists of solutions inside a moduli space of the form
\[\mathcal{M}^J_{L \cup L}(a';\mathbf{b}')=\mathcal{M}^J_L(a';\mathbf{b}')\]
which is transversely cut-out by assumption and, in particular, of non-negative index.

Second, by Lemma \ref{lem:strips}, the components in the $\alpha$ and $\gamma$-levels occurring in the limit are of non-negative index as well. Since thus every component occurring in the limit building is of non-negative index, and since the indices sum to zero by assumption, every component in the building is of index zero. It thus follows that every component appearing in the $\alpha$ and $\gamma$-layers is of the form described in Lemma \ref{lem:strips}.

In conclusion, we have shown that the possible buildings appearing as a limit when stretching the neck of the moduli space on the left-hand side consists of levels of the following form:
\begin{itemize}
\item The $\alpha$-level consisting of:
\begin{itemize}
\item The trivial strips $\R \times \{c_i\}$;
\item The unique strips in $\mathcal{M}^{\widetilde{J}_P}_{\phi^\epsilon_{h_-}(\R \times \Lambda^-)}(d_i;d_i^\epsilon)$ which are embedded and contained inside the planes $\R \times \{p_{d_i}\} \times \R$; and
\item The unique strip in $\mathcal{M}^{\widetilde{J}_P}_{\Lambda^-,\phi^\epsilon_{h_-}(\R \times \Lambda^-)}(b;b^\epsilon)$ which is embedded and contained inside the plane $\R \times \{p_b\} \times \R$;
\end{itemize}
\item The $\beta$-level consisting of a rigid solution in $\mathcal{M}^J_L(a;\mathbf{c}b\mathbf{d})$; and
\item The $\gamma$-level consisting of the unique strip in $\mathcal{M}^{\widetilde{J}_P}_{\Lambda^+,\phi^\epsilon_{h_+}(\R \times \Lambda^+)}(a^\epsilon;a)$ which is embedded and contained inside the plane $\R \times \{p_a\} \times \R$.
\end{itemize}

The sought bijection of solutions now follows by pseudo-holomorphically gluing buildings of the above type. To be able to glue, we must assume that the building consists of pseudo-holomorphic discs that all are transversely cut out solutions. For the solution in the $\beta$-level, this holds by assumption. The transversality of the solutions contained in the $\alpha$ and $\gamma$-levels was shown in Lemma \ref{lem:strips}.
\end{proof}

\begin{Lem}
\label{lem:strips}
Let $J_P$ be a regular compatible almost complex structure on $P$ and let $\widetilde{J}_P$ be its cylindrical lift. The solutions in
\begin{gather*}
\mathcal{M}^{\widetilde{J}_P}_{\Lambda^+,\phi^\epsilon_{h_+}(\R \times \Lambda^+)}(a^\epsilon;\mathbf{c}',b',\mathbf{d}'),\\
\mathcal{M}^{\widetilde{J}_P}_{\Lambda^-,\phi^\epsilon_{h_-}(\R \times \Lambda^-)}(a';\mathbf{c}',b^\epsilon,\mathbf{d}'),\\
\mathcal{M}^{\widetilde{J}_P}_{\phi^\epsilon_{h_-}(\R \times \Lambda^-)}(d_i;\mathbf{d}^\epsilon),
\end{gather*}
which are of non-positive index are in bijective correspondence with the unique embedded solutions contained inside planes of the form $\R \times \{p_c\} \times \R$, where $p_c \in \Pi_{\operatorname{Lag}}(\Lambda^\pm)$ is the double-point corresponding to a Reeb chord  $c \in \mathcal{Q}(\Lambda^\pm)$. In particular, these solutions satisfy
\[ \mathbf{c}'=\mathbf{d}'=0, \:\: b'=b, \:\: a'=b, \:\: \mathbf{d}^\epsilon=d^\epsilon,\]
are of index zero and, moreover, are transversely cut-out.
\end{Lem}
\begin{proof}
The canonical projection $\R \times P \times \R \to P$ is $(\widetilde{J}_P,J_P)$-holomorphic and projects the cylinder $\phi^\epsilon_{h_\pm}(\R \times \Lambda^\pm)$ to the Lagrangian projection $\Pi_{\operatorname{Lag}}(\Lambda^\pm) \subset P$. Any $\widetilde{J}_P$-holomorphic disc as above must thus project to a $J_P$-holomorphic disc in $P$ having boundary on $\Pi_{\operatorname{Lag}}(\Lambda^\pm)$. Moreover, since the index of the projected curve is one less than the index of the original curve by \cite[Lemma 8.3]{Lifting}, the regularity of $J_P$ implies that the projection must be constant.

The above solutions can thus be described explicitly as follows. For a Reeb chord $c$ on $\Lambda^\pm \cup \Lambda^\pm_\epsilon \subset P \times \R$, we write $p_c \in P$ for the unique point in its canonical projection to $P$. In other words, we have shown that a solution as above is contained inside the plane $\R \times \{p_c\} \times\R$, where $c$ denotes its positive puncture. It follows that the solution is an embedded strip having exactly one positive and one negative puncture.

Finally, the index calculation is standard, while the transversality was shown in \cite[Lemma 8.3]{Lifting}.
\end{proof}

\begin{Lem}
\label{lem:breaking}
Let $J$ be as in Theorem \ref{thm:pushoff}. Given a solution
\[u \colon (\R \times [0,1],\R \times \{0\},\R \times \{1\} \to (\R \times P \times \R,L,\phi^\epsilon_{h_N}(L'))\]
contained inside
\[\mathcal{M}^J_{L,\phi^\epsilon_{h_N}(L')}(a^\epsilon;\mathbf{c},b^\epsilon,\mathbf{d}^\epsilon),\]
where either $L'=L$ or $L'$ is a generic compactly supported perturbation of $L$, there is an a-priori non-zero lower bound on the $C^0$-norm of any path $s \mapsto u(t_0,s)$, given that $\epsilon>0$ is chosen sufficiently small. In particular, the usual SFT compactness theorem can be applied to these moduli spaces.
\end{Lem}
\begin{proof}
The difficulty lies in the fact that these boundary-value problems are not generic; for instance, the intersections $L \cap \phi^\epsilon_{h_N}(L)$ are not transverse and, even worse, the intersections $I_{\operatorname{bad}}:=\partial(L \cap \phi^\epsilon_{h_N}(L))$ are not even of Bott-type. Since Gromov-Floer compactness has not yet been established for intersections of this degenerate form, we have to argue as follows. We restrict attention to the case $L'=L$, the argument is similar in the other cases.

The projection of any solution $u \in \mathcal{M}^J_{L,\phi^\epsilon_{h_N}(L)}(a^\epsilon;\mathbf{c},b^\epsilon,\mathbf{d}^\epsilon)$ to $P$ restricted to $u^{-1}\{ |t| \ge M \}$ defines a $J_P$-holomorphic curve having boundary on $\Pi_{\operatorname{Lag}}(\Lambda^\pm)$ by the choice of almost complex structure. Observe that the projected boundary-condition is in fact generic.

We can thus apply the Gromov-Hofer-Floer compactness theorem to $u$, with the caveat that, for a hypothetical Floer breaking occurring near $I_{\operatorname{bad}} \subset \{ |t| \ge M \}$, the end of the strip might only have a well-defined $P$-component (for which the compactness theorem can be applied).

Luckily, the above strips (parts of which a priori only have a well-defined $P$-component) can be seen to still satisfy Formula \eqref{eq:energy2}. This enables us to exclude Floer-breakings at intersection points all together. The reason is that the definition of the $d\alpha$-energy of a curve is computed entirely in terms its projection to $P$ outside of $\{ |t| \ge M \}$. If a breaking occurred at  an intersection point, we would now be able to find a strip (which might partly only have a well-defined $P$-component) having its negative punctures asymptotic to Reeb chords while its positive puncture is mapped to an intersection point. However, such a strip would necessarily have negative $d\alpha$-energy, given that $\epsilon>0$ is sufficiently small. The latter fact can be seen by applying Formula \eqref{eq:energy2} together with estimate $|\mathfrak{a}(p)| \le C\epsilon$ which holds for some constant $C>0$ depending only on $L$.

The above contradiction shows that no Floer breaking is possible. The claimed $C^0$-bound now follows.
\end{proof}

\section{The proof of Theorem \ref{thm:main}}
\label{sec:proof-theor-refthm:m}
Let $(P, d \theta)$ be a Liouville manifold and $(P \times \R, dz + \theta)$ its contactisation.
We let $V_i \subset \R \times P \times \R$, $i=0,1$, denote the exact fillings of the Legendrian submanifold $\Lambda^- \subset P \times \R$ inducing the augmentations $\varepsilon_i$, $i=0,1$, and $L  \subset \R \times P \times \R$ the exact Lagrangian concordance from $\Lambda^-$ to the Legendrian submanifold $\Lambda^+\subset P \times \R$. We write $V_i \odot L$ for the exact Lagrangian filling of $\Lambda^+$ induced by concatenating the filling $V_i$ with the concordance $L$. More precisely, after a translation we may assume that
\begin{gather*}
V_i \cap \{ t \ge -1\} = [-1,+\infty) \times \Lambda^-,\\
L \cap \{ t \le 1 \} = (-\infty,-1] \times \Lambda^-.
\end{gather*}
and we define $V_i \odot L$ to be the exact Lagrangian submanifold which coincides with $V_i$ on $\{ t \le 0\}$ and with $L$ on $\{ t \ge 0\}$.

The assumption that $L$ is a concordance implies that the restriction $e^t|_L$ can be modified inside a compact set to a smooth function without singular points. For an appropriate extension of this function to a Hamiltonian $H \colon \R \times P \times \R \to \R$, the corresponding Hamiltonian isotopy $\phi^t_H$ can be made to satisfy
\begin{itemize}
\item $\phi^t_H$ is the flow of the Reeb vector field $\partial_z$ on the subset
\[ \{ -1 \le t \le 1\} \cup \{ t \ge N \} \subset \R \times P \times \R, \]
for some sufficiently large $N>0$; and
\item $\phi^\epsilon_H (L) \cap L \cap \{ t \ge -1\} =\emptyset$ for each $\epsilon>0$ sufficiently small.
\end{itemize}
Finally, we let $\phi^t_G$ be the Hamiltonian isotopy which coincides with $\phi^t_H$ on $\{ t \le 1\}$ while it is generated by $\partial_z$ on $\{t \ge -1\}$. We will compute the wrapped Floer homology complexes
\[(CW_\bullet(V_0 \odot L,\phi^\epsilon_H(V_1 \odot L)),d)\]
and
\[(CW_\bullet(V_0,\phi^\epsilon_G(V_1)),d')\]
as describe in Section \ref{wfh}.

Write $\Lambda^\pm_\epsilon$ for the time-$\epsilon$ Reeb flow of $\Lambda^\pm$ (i.e. translation by $\epsilon$ in the $z$-coordinate). Observe that there is a canonical bijection between the Reeb chords starting on $\Lambda^\pm_\epsilon$ and ending on $\Lambda^\pm$ and the Reeb chords on $(\Lambda^\pm)$ given that $\epsilon>0$ is chosen sufficiently small. By the definition of the wrapped Floer homology complex, we thus have the decomposition
\begin{gather*}
CW_\bullet(V_0,\phi^\epsilon_G(V_1))=LCC^{\bullet-1}(\Lambda^-)\oplus \Z_2(V_0 \cap \phi^\epsilon_G(V_1)) ,\\
d'=\begin{pmatrix} d'_\infty & \delta' \\
0& d'_0
\end{pmatrix}.\end{gather*}
Using Proposition \ref{prop:identification} together with Theorem \ref{thm:pushoff}, assuming that we have chosen the compatible almost complex structure in the definition of $d'$ appropriately, we conclude that
\[d'_\infty=d_{\varepsilon_0,\varepsilon_1} \]
is the bilinearised Legendrian contact cohomology differential for $LCC^\bullet_{\varepsilon_0,\varepsilon_1}(\Lambda^-)$. Recall that $\varepsilon_i$ is the augmentation induced by $V_i$, $i=0,1$.

By the construction of $\phi^t_H$ it follows that
\[V_0 \odot L \: \cap \: \phi^\epsilon_H(V_1 \odot L)=V_0 \cap \phi^\epsilon_G(V_1),\]
from which we get the decomposition
\begin{gather*}
CW_\bullet(V_0 \odot L,\phi^\epsilon_H(V_1 \odot L))=LCC^{\bullet-1}(\Lambda^+)\oplus \Z_2(V_0 \cap \phi^\epsilon_G(V_1)) ,\\
d=\begin{pmatrix} d_\infty & \delta \\
0& d_0
\end{pmatrix}.\end{gather*}
Again, after using Proposition \ref{prop:identification} together with Theorem \ref{thm:pushoff}, we may assume that
\[d_\infty=d_{\varepsilon_0 \circ \Phi_L,\varepsilon_1  \circ \Phi_L} \]
is the bilinearised Legendrian contact cohomology differential on
\[LCC^\bullet_{\varepsilon_0 \circ \Phi_L,\varepsilon_1  \circ \Phi_L}(\Lambda^+).\]

After stretching the neck around the hypersurface $\{ t =0\} \subset \R \times P \times \R$ of contact type as in Section \ref{sec:neckstretch}, we can moreover deduce that $d$ is of the following form. First, since there are no pseudo-holomorphic discs with boundary on an embedded exact Lagrangian cobordism without positive punctures, we may assume that
\[ d_0=d'_0.\]
Second, using Theorem \ref{thm:pushoff} together with an analysis of the possible broken curves under the above neck-stretching sequence, we deduce that there is a factorisation
\begin{equation}
\label{eq:factor}
\delta=\delta' \circ \Phi^{\varepsilon_0,\varepsilon_1}_L,
\end{equation}
where we recall that both $\delta$ and $\delta'$ are chain maps.

\begin{Rem}The factorisation in Formula \eqref{eq:factor} is a special case of the so-called transfer map in wrapped Floer homology. We refer to \cite[Section 4.2.2]{RationalSFT2} and \cite[Section 5.2]{Lifting} for more details.
\end{Rem}

Finally, the invariance properties of wrapped Floer homology, see Proposition \ref{prop:invariance}, implies that the chain maps
\begin{gather*} \delta \colon (\Z_2(V_0 \cap \phi^\epsilon_G(V_1)),d'_0) \to (LCC^\bullet_{\varepsilon_0 \circ \Phi_L,\varepsilon_1  \circ \Phi_L}(\Lambda^+),d_{\varepsilon_0 \circ \Phi_L,\varepsilon_1  \circ \Phi_L}) \\
\delta' \colon (\Z_2(V_0 \cap \phi^\epsilon_G(V_1)),d'_0) \to (LCC^\bullet_{\varepsilon_0 ,\varepsilon_1  }(\Lambda^-),d_{\varepsilon_0 ,\varepsilon_1})
\end{gather*}
both are quasi-isomorphisms. By the factorisation \eqref{eq:factor} it thus follows that the chain map
\[  \Phi^{\varepsilon_0,\varepsilon_1}_L \colon (LCC^\bullet_{\varepsilon_0,\varepsilon_1}(\Lambda^-),d_{\varepsilon_0,\varepsilon_1})\to (LCC^\bullet_{\varepsilon_0 \circ \Phi_L,\varepsilon_1  \circ \Phi_L}(\Lambda^+),d_{\varepsilon_0 \circ \Phi_L,\varepsilon_1  \circ \Phi_L})\]
is a quasi-isomorphism as well.
\qed

\section{Lagrangian concordances of the Legendrian unknot}\label{sec:lagr-conc-legendr}
\subsection{Classification of two-dimensional concordances}
The Legendrian unknot $\Lambda_0  \subset (S^3,\xi_{st})$ of ${\tt tb}(\Lambda)=-1$ can be represented as the intersection of $S^3 \subset \C^2$ with the real part $\R^2$ of $\C^2$. Since $\C^2 \setminus \{0\}$ is symplectomorphic to the symplectisation of $(S^3,\xi_{st})$, slightly deforming the Lagrangian plane $\R^2\subset\C^2$ in a neighbourhood of the origin produces an exact Lagrangian filling of $\Lambda_0$ inside the symplectisation of $(S^3, \xi_{st})$. Eliashberg and Polterovich proved the following strong classification result.
\begin{Thm}[Eliahberg-Polterovich \cite{LocalLagrangian}]
\label{thm:eliashbergpolterovich}
There is a unique exact orientable Lagrangian filling of $\Lambda_0 \subset (S^3,\xi_{st})$ in $\C^2$ up to compactly supported Hamiltonian isotopy. In particular, such an exact Lagrangian filling is a plane (and hence its compact part is a disc).
\end{Thm}
\begin{Rem}
Ritter has shown that every exact filling of $\Lambda_0 \subset (S^3,\xi_{st})$ is orientable \cite[Corollary 14]{RitterNovikov}.
\end{Rem}

We will use this fact to prove a classification of exact Lagrangian cobordisms from $\Lambda_0$ to $\Lambda_0$.

First, let $\rho \colon \R_{\ge 0} \to \R_{\ge 0}$ be a smooth function that vanishes in a neighbourhood of the origin  and satisfies $\rho'(t) \ge 0$, while $\rho(t)=1$ holds outside of a compact set. Consider the (non-compactly supported) diffeomorphism
\begin{gather*}\phi \colon \C^2 \setminus \{0\} \to \C^2 \setminus \{ 0\},\\
(z_1,z_2) \mapsto (e^{i \pi \rho( \|\mathbf{z}\|) }z_1,z_2).
\end{gather*}

\begin{Lem}
\label{lem:wrap}
The map $\phi$ is a symplectomorphism of $\C^2 \setminus \{0\}$ which fixes $\R \times \Lambda_0 \subset \R \times S^3 \simeq \C^2 \setminus \{0\}$ set-wise outside of a compact set, and whose even powers have compact support. Moreover, $\phi^l(\R \times \Lambda_0)$ and $\phi^m(\R \times \Lambda_0)$, $m,l \in \Z$, are compactly supported smoothly isotopic if and only if $l$ and $m$ have the same parity. Finally, $\phi^l(\R \times \Lambda_0)$ and $\phi^m(\R \times \Lambda_0)$ are compactly supported Hamiltonian isotopic if and only if $l=m \in \Z$.
\end{Lem}
\begin{proof}
Observe that $U(2)$ acts on $(S^3,\alpha_{st})$ preserving the contact form. The first claim follows since $\phi$ preserves the concentric spheres $S^3_R \subset \C^2$ of radius $R$ set-wise, while it acts on each $S^3_R$ by an element in $U(2)$.

The claim concerning the smooth isotopy can be seen as follows. First, observe that the images of $\R \times \Lambda_0$ under powers of $\phi$ are given as the trace of a one-parameter family of smooth embeddings of the oriented unknot inside $S^3$ which, after a smooth isotopy, may be assumed to be induced by a one-parameter family of oriented two-planes inside $\R^3$ where the knot is the unit-circle in the two-plane. The statement follows from the fact that $\pi_1(Gr_2(\R^4)) \simeq \Z_2$ (see e.g.~\cite[10.8.C]{TopologyII}).

For the last statement, it suffices to compute the Maslov number of the path
\[[-N-1,N+1] \ni t \mapsto \phi^m(t,y) \subset \phi^m(\R \times \Lambda_0),\]
for some fixed point $y \in \Lambda_0$. This number, which can be computed to be equal to $m$, is invariant under Hamiltonian isotopies supported in the set $\{ |t| \le N \}$.
\end{proof}

\begin{Thm}
\label{thm:concisotopy}
Let $\Lambda_0 \subset (S^3, \xi_{st})$ be a Legendrian unknot with ${\tt tb}(\Lambda_0)=-1$, and let $L$ be an exact Lagrangian cobordism from $\Lambda$ to itself inside the symplectisation of $(S^3, \xi_{st})$. It follows that $L$ is compactly supported Hamiltonian isotopic to $\phi^m(\R \times \Lambda_0)$ for some $m \in \Z$.
\end{Thm}
\begin{Rem}
\begin{enumerate}
\item\label{rem:unknot} The space of smooth unparametrised embeddings of the unknot in $S^3$ is homotopy equivalent to $Gr_2(\R^4)$ as follows by the positive answer of the Smale conjecture by Hatcher \cite[Appendix]{SmaleConjecture}, i.e.~the fact that the inclusion $O(4) \subset \operatorname{Diff}(S^3)$ is a homotopy equivalence. In other words, since $\pi_1(Gr_2(\R^4)) \simeq \Z_2$ (see e.g.~\cite[10.8.C]{TopologyII}), the space of \emph{unparametrised oriented} unknots is simply connected. Theorem \ref{thm:concisotopy} suggests that the fundamental group of unparametrised oriented Legendrian knots is isomorphic to $\Z$. Recall that a Legendrian isotopy induces a Lagrangian concordance by \cite{LagrConc}. Finally, note that this question is related to a result by Sp{\'a}{\v{c}}il, which states that the inclusion $U(2) \subset \operatorname{Cont}(S^3;\xi_{st})$ is a homotopy equivalence. On Figure \ref{fig:non_triv_loop} we show the Lagrangian and front projection of the isotopy inducing the cylinder $\phi(\R \times \Lambda_0)$, this is a non-trivial loop of Legendrian submanifold which can be distinguished using the Maslov class.
\item Using the non-contractible loops of the non-trivial torus knots produced in \cite{OneParamKnots}, one can produce plenty of examples of exact Lagrangian concordances from a non-trivial torus-knot to itself which are in different compactly supported Hamiltonian isotopy classes but which cannot be distinguished by the Maslov class nor by any smooth invariant.
\end{enumerate}
\end{Rem}

\begin{figure}[ht!]
  \centering
  \includegraphics[width=\textwidth]{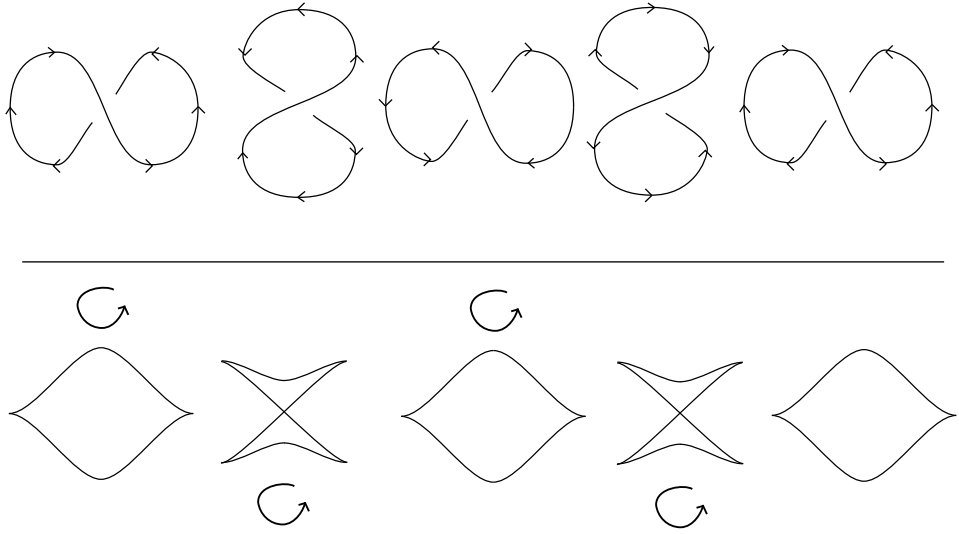}
  \caption{Lagrangian (at the top) and front (at the bottom) projections of a non-trivial loop of Legendrian unknots induced by the $2\pi$-rotation of the $x,y$-plane.}
  \label{fig:non_triv_loop}
\end{figure}

In order to prove Theorem \ref{thm:concisotopy} we will need the following technical Lemmas concerning Hamiltonian isotopies, which are all standard.

\begin{Lem}
\label{lem:flowpoints}
For any smooth one-parameter family $f_t \colon I^k \to (X,\omega)$ of smooth maps, there is an induced smooth family $\Phi^t_{\mathbf{s}} \in \operatorname{Ham}^c(X)$, $\mathbf{s} \in I^k$, of smooth maps satisfying
\begin{itemize}
\item $\frac{d}{dt}\Phi^t_{\mathbf{s}}(f_{t_0}(\mathbf{s}))|_{t=t_0}=\frac{d}{dt} f_t(\mathbf{s})|_{t=t_0} \in T_{f_{t_0}(\mathbf{s})}X$;
\item $\frac{d}{dt}\Phi^t_\mathbf{s} \in \operatorname{Ham}^c(X)$ is supported in an arbitrarily small neighbourhood of $f_t(\mathbf{s}) \in X$; and
\item $\frac{d}{dt}\Phi^t_{\mathbf{s}}|_{t=t_0} \equiv 0$ for each $\mathbf{s} \in I^k$ satisfying $\frac{d}{dt} f_t(\mathbf{s})|_{t=t_0}=0$.
\end{itemize}
\end{Lem}
\begin{proof}
A $(k+1)$-parameter family version of Darboux' theorem \cite{SympTop} produces a smooth family of symplectomorphisms
\[ \phi_{t,\mathbf{s}} \colon (B^{2n}_\epsilon,\omega_0) \to (X,\omega), \:\: \phi_{t,\mathbf{s}}(0)=f_t(\mathbf{s}).\]
 Pick a smooth compactly supported bump-function $\rho \colon B^{2n}_\epsilon \to \R_{\ge 0}$ which is equal to one in some neighbourhood of the origin, and whose support is sufficiently small. We define $\Phi^t_{\mathbf{s}}$ by the requirement that, in the Darboux coordinates $\phi_{t,\mathbf{s}}$, at time $t_0$ it is the Hamiltonian flow induced by $\rho \cdot H_{t_0}$ for the \emph{linear} Hamiltonian $H_{t_0} \colon B^{2n}_\epsilon \to \R$ generating the flow of $\frac{d}{dt} f_t(\mathbf{s})|_{t=t_0}$ at the origin.
\end{proof}

\begin{Lem}
\label{lem:makelinear}
Let $\phi_t \colon (B^{2n}_\epsilon,\omega_0) \to (\C^n,\omega_0)$ be a one-parameter family of symplectomorphisms fixing the origin, and assume that $\phi_0=\phi_1=\id_{B^{2n}_\epsilon}$. Then there is a family $\phi_{t,s}  \colon (B^{2n}_\epsilon,\omega_0) \to (\C^n,\omega_0)$ of symplectomorphisms fixing the origin and satisfying
\renewcommand {\theenumi}{\roman{enumi}}
\begin{enumerate}
\item $\phi_{t,0}=\phi_t$ and $D_0\phi_{t,s}=D_0\phi_t$;
\item $\phi_{t,1}=D_0\phi_t$ in some neighbourhood of the origin;
\item $\phi_{t,s}=\phi_t$ outside of a compact set; and
\item $\phi_{0,s}=\phi_{1,s}=\id_{B^{2n}_\epsilon}$;
\end{enumerate}
\end{Lem}
\begin{proof}
We define a smooth family of symplectomorphisms $\psi_{t,s}:=\phi_t(s \mathbf{x})/s$, $s \in (0,1]$. Considering the Taylor expansion of the form
\[ \phi_t(\mathbf{x}) = D_0\phi_t(\mathbf{x})+\mathbf{x}^{\operatorname{t}}Q_t\mathbf{x} +\mathbf{x}^{\operatorname{t}}K_t(\mathbf{x})\mathbf{x} ,\]
where $Q_t$ is a symmetric matrix, and $K_t$ depends smoothly on $(t,\mathbf{x})$ and vanishes at $\mathbf{x}=0$, it becomes clear that $\psi_{t,s}$ is smooth for $(t,s) \in [0,1] \times [0,1]$.

Observe that $\psi_{t,s}$ satisfies (i), (ii) and (iv) by construction. We can deform this path of functions to make it satisfy (iii) as well. To that end, observe that the paths $s \mapsto \psi_{t,s}\circ \phi_t^{-1}$ of symplectomorphisms fixing the origin (which all are defined in some neighbourhood of the origin) can be generated by Hamiltonian isotopies, i.e.~
\[\psi_{t,s}\circ \phi_t^{-1}=\phi^s_{H_{t,s}}\]
for some Hamiltonians $H_{t,s} \colon B^{2n}_\epsilon \to \R$. Choosing a suitable compactly supported bump function $\rho \colon B^{2n}_\epsilon \to \R_{\ge 0}$ equal to one in some neighbourhood of the origin, the sought family can be taken to be
\[ \phi_{t,s} := \phi^s_{\rho H_{t,s}} \circ \phi_t.\]
\end{proof}

\begin{Lem}
\label{lem:makelinearlag}
Suppose that $\phi \in Symp^c(T^*M)$ fixes the zero-section $0_M \subset T^*M$ set-wise. It follows that there exists a compactly supported Hamiltonian isotopy $\phi^t_{H_t}$ for which
\[\phi^1_{H_t} \circ \phi=(\phi|_{0_M}^{-1})^*\]
holds in some neighbourhood of $0_M$, where $\phi^t_{H_t}$ moreover can be taken to be the identity along the zero-section.
\end{Lem}
\begin{proof}
After pre-composing $\phi$ with $(\phi|_{0_M})^*$, we see that it suffices to consider the case when $\phi|_{0_M}=\id_{0_M}$. After a Hamiltonian isotopy, we may furthermore assume that $D\phi$ is the identity along the zero-section. The statement now follows similarly to Lemma \ref{lem:makelinear}, after the observation that $\psi_s:=\phi(s\mathbf{x})/s$, $s \in [0,1]$, is smooth a family of symplectomorphisms defined in some neighbourhood of the zero-section (where the multiplication is fibre-wise).
\end{proof}

\begin{proof}[Proof of Theorem \ref{thm:concisotopy}]
We view the concordance as an exact Lagrangian submanifold in $\C^2 \setminus \{0\}$ which coincides with the real part outside of a compact set. We can clearly remove the singularity at the origin, thus producing an exact Lagrangian filling of $\Lambda_0$. It follows from Theorem \ref{thm:eliashbergpolterovich} that there is a compactly supported Hamiltonian isotopy $\phi^t_{H_t} \colon \C^2 \to \C^2$ for which the image of this filling under the time-one map is $L_0 \subset \C^2$. The goal is to deform this isotopy to make it fix a neighbourhood of $0 \in \C^2$.

First, by Lemma \ref{lem:makelinearlag}, there exists a compactly supported Hamiltonian isotopy fixing $L_0$ after which we may assume that $\phi^1_{H_t}$ coincides with the linear symplectomorphism
\[\begin{pmatrix}\pm 1 & 0 \\
0 & 1 \end{pmatrix} \in U(2)\]
in some neighbourhood of the origin. Second, using lemmas \ref{lem:flowpoints} and \ref{lem:makelinear}, we can produce a family of compactly supported symplectomorphisms $\psi_t \colon \C^2 \to \C^2$, $\psi_0=\psi_1=\id_{\C^2}$, which make $\Phi_t:=\psi_t \circ \phi^t_{H_t}$ into a path of compactly supported symplectomorphisms which all are linear in some neighbourhood of the origin (and hence, in particular, fix the origin).

Since the inclusion $U(n) \subset \operatorname{Sp}(n)$ is a homotopy equivalence, there exists a loop $A_t \in \operatorname{Symp}^c(\C^2)$, $A_0=A_1=\id_{\C^2}$, of compactly supported symplectomorphisms which are linear in a neighbourhood of the origin, for which the equality
\[A_t \circ \Phi_t= \begin{pmatrix}e^{i\pi m t} & 0 \\
0 & 1
\end{pmatrix},\]
is satisfied for some $m \in \Z$.

The sought isotopy can now be constructed by considering the above family $A_t \circ \Phi_t$ of symplectomorphisms composed with a family
\begin{gather*}B_t \colon \C^2 \setminus \{0\} \to \C^2 \setminus \{ 0\},\\
(z_1,z_2) \mapsto (e^{-i \pi m \sigma( \|\mathbf{z}\|)t }z_1,z_2),
\end{gather*}
of symplectomorphisms (see Lemma \ref{lem:wrap}) where $\sigma \colon \R_{\ge 0} \to \R_{\ge 0}$ is a smooth bump-function equal to one in a neighbourhood of the origin, satisfies $\sigma'(t) \le 0$, while $\sigma(t)=0$ holds outside of a compact set. Observe that $B_t$ coincides with the identity outside of a compact set, while it is equal to $A_t^{-1}$ in a neighbourhood of the origin.
\end{proof}

\subsection{Higher dimensional generalisation of the Eliashberg-Polterovich result}
There is no analogous classification of exact Lagrangian fillings of the unknot, nor of concordances, in higher dimensions. Using wrapped Floer homology  we can however obtain strong topological restrictions on such fillings. Theorem \ref{malgthinstregbam} gives information on the homology of the filling, while a version of wrapped Floer homology with local coefficients can, sometimes, give information on
the fundamental group.

Local coefficients in Floer homology was first considered by Damian in \cite{FloerHomUniv}. In the current setting, this is related to the version of Legendrian contact homology with local coefficients developed in \cite{Albin}. We also refer to \cite[Section 2]{EkholmSmith} for similar techniques.
\begin{Thm}
\label{thm:fukayaseidelsmith}
Let $\Lambda \subset (\R^{2n+1},\xi_{st})$ be Legendrian submanifold having a single transverse Reeb chord (e.g.~we can take $\Lambda=\Lambda_0$). Any exact Lagrangian filling $L$ of $\Lambda$ is contractible and, given that such a filling exists, it follows that $\Lambda$ is a homotopy sphere. In the latter case, $(\overline{L},\partial\overline{L})$ is diffeomorphic to $(D^{n+1},\partial D^{n+1})$ given that $n \ne 3,4$. For $n = 4$ the same is true under the further assumption that $\Lambda$ is diffeomorphic to $S^4$. If $n=3$, then $(\overline{L},\partial\overline{L})$ is {\em homeomorphic} to $(D^{4},\partial D^{4})$.

\end{Thm}
\begin{Rem}
\label{rem:bull}
When $\Lambda=\Lambda_0$, and under the additional assumption that the Maslov class of $L$ vanishes, the above statement follows from a result due to Abouzaid \cite{NearbyLagrangian}, based upon earlier work by Fukaya-Seidel-Smith \cite{ExactLagrangianCotangent} and Nadler \cite{MicrolocalBranes}. To that end, observe that a Weinstein handle-attachment on $(B^{2n+2},\omega_0)$ along $\Lambda_0 \subset \partial B^{2n+2}$ produces the symplectic manifold $D^*S^{n+1}$ and that the union of $L$ with the Lagrangian core of the corresponding handle is a closed exact Lagrangian submanifold. Finally, the assumption on the Maslov class was removed in \cite{RingSpectra}.
\end{Rem}


Let $L$ be an exact Lagrangian filling of $\Lambda$. We denote by $\widetilde{L}$ the
universal cover of $L$ and by $\widetilde{\Lambda}$ the preimage of $\Lambda$ to $\widetilde{L}$. Two important points should be stressed here: the first one is that $\widetilde{\Lambda}$ is not the universal cover of $\Lambda$ in general; in fact
$\widetilde{\Lambda}$ does not have to be either connected or simply connected. The
second one is that we do not claim that $\widetilde{L}$ embeds in $\R^{2n+2}$ as a Lagrangian submanifold, nor that $\widetilde{\Lambda}$ embeds in  $\R^{2n+1}$ as a
Legendrian submanifold.

We fix a base point $* \in L$. The fundamental group $\pi_1(L, *)$ acts on $\widetilde{L}$
(and therefore on $\widetilde{\Lambda}$) by deck transformations.
We denote by $\mathfrak{p} \colon \widetilde{L} \to L$ the projection and $\pi_1(L, *)= \pi$.
We define the $\Z_2$-vector space $\underline{CF}_\bullet(L, \phi^\epsilon_{e^t}(L))$  generated by the points $\tilde{p} \in \widetilde{L}$ such that $\mathfrak{p}(\tilde{p})
\in L \cap \phi^\epsilon_{e^t}(L)$. The degree of $\widetilde{p}$ is, by definition, the degree
of $\mathfrak{p}(\tilde{p})$ as a generator of $CF_\bullet(L, \phi^\epsilon_{e^t}(L))$. The group $\pi$ acts on $\underline{CF}_\bullet(L, \phi^\epsilon_{e^t}(L))$ by deck transformations. This action turns $\underline{CF}_\bullet(L, \phi^\epsilon_{e^t}(L))$ into a  free $\Z_2[\pi]$-module, and thus
$$\underline{CF}_\bullet(L, \phi^\epsilon_{e^t}(L)) \cong CF_\bullet(L, \phi^\epsilon_{e^t}(L)) \otimes_{\Z_2} \Z_2[\pi]$$
as a $\Z_2$-vector space. The choice of such an
isomorphism is equivalent to the choice of a base point in $\mathfrak{p}^{-1}(p)$ for
all $p \in L \cap \phi^\epsilon_{e^t}(L)$.

From now on we will write ${\mathcal M}(p, q):={\mathcal M}_{L, \phi^\epsilon_{e^t}(L)}(p; q)$ for simplicity. Given generators $\tilde{p}$ and $\tilde{q}$ of $\underline{CF}_\bullet(L, \phi^\epsilon_{e^t}(L))$ with $\mathfrak{p}(\tilde{p})=p$ and $\mathfrak{p}(\tilde{q})=q$ and a strip $u \in {\mathcal M}(p, q)$ we define
$\ell_{u, \tilde{p}} \colon \R \to \widetilde{L}$ as the lift of $u|_{\R \times {0}}$
to $\widetilde{L}$ such that $\lim \limits_{s \to - \infty} \ell_{u, \tilde{p}}(s) = \tilde{p}$. We will use the map
$\ell_{u, \tilde{p}}$ as a mean to keep track of the homotopy class of the strip $u$.
We define
$$\underline{\mathcal M}(\tilde{p}, \tilde{q}) \subset {\mathcal M}(p, q)$$
as the set of holomorphic strips $u \in {\mathcal M}(p, q)$ such that
$\lim \limits_{s \to + \infty} \ell_{u, \tilde{p}}(s) = \tilde{q}$.
We denote by  $\# \underline{\mathcal M}(\tilde{p}, \tilde{q})$ the count modulo two of strips in  $\# \underline{\mathcal M}(\tilde{p}, \tilde{q})$ when those space are $0$-dimensional else we set it to $0$. (If
$L$ is spin and a coherent orientation system for the moduli spaces ${\mathcal M}(p,q)$ is chosen, we can restrict it to $\underline{\mathcal M}(\tilde{p}, \tilde{q})$ and define
the count over the integers.)

By a routine exercise on covering spaces, the $\Z_2$-linear map
$$\underline{d}_0 \colon \underline{CF}_\bullet(L, \phi^\epsilon_{e^t}(L)) \to \underline{CF}_{\bullet+1}(L, \phi^\epsilon_{e^t}(L))$$
defined on the generators by
$$\underline{d}_0 (\tilde{q}) = \sum \limits_{\tilde{p} \in \mathfrak{p}^{-1}(L \cap \phi^\epsilon_{e^t}(L))} \# \underline{\mathcal M}(\tilde{p}, \tilde{q}) \tilde{p}$$
is a $\Z_2[\pi]$-linear differential. We denote by $\underline{HF}_\bullet(L, \phi^\epsilon_{e^t}(L))$ its homology. The correspondence between $J$-holomorphic strips in ${\mathcal M}(p, q)$ and gradient flow trajectories on $L$ (see \cite{MorseTheoryLagr}) shows that there is a $\Z_2[\pi]$-linear isomorphism
$$\underline{HF}_\bullet(L, \phi^\epsilon_{e^t}(L)) \cong H_{n- \bullet}(\widetilde{L};\Z_2)$$
for the natural $\pi$-action on the Morse homology of $\widetilde{L}$.

Below we will use the identification $\Lambda \cong \partial\overline{L} \subset L$. Now we define the twisted version of linearised contact cohomology. Given a Reeb chord $a$ from $\phi^\epsilon(\Lambda)$ to $\Lambda$ we denote by $a_e$ the endpoint of $a$ on $\Lambda$, and use $\widetilde{\mathcal{Q}}(\phi^\epsilon(\Lambda),\Lambda)$ to denote the set of all pairs $\tilde{a} = (a, x)$, where $a$ is a Reeb chord from $\phi^\epsilon(\Lambda)$ to $\Lambda$ and $x \in \mathfrak{p}^{-1}(a_e)$. We define $\underline{LCC}_{\varepsilon_0,\varepsilon_1}^\bullet(\Lambda,\phi^\epsilon(\Lambda))$ as the $\Z_2$-module generated $\widetilde{\mathcal{Q}}(\Lambda,\phi^\epsilon(\Lambda))$. The degree of $\widetilde{a}$ is, by definition, the degree of $a$. Here we should consider $\varepsilon_0$ and $\varepsilon_1$ as augmentations induced by $L$ and $\phi^\epsilon_{e^t}(L)$, respectively, where the count of one-punctured pseudo-holomorphic discs with boundary on $\phi^\epsilon_{e^t}(L)$ keeps track of the homotopy class of the discs (see the definition of the differential below). The group $\pi$ acts on $\underline{LCC}_{\varepsilon_0,\varepsilon_1}^\bullet(\Lambda,\phi^\epsilon(\Lambda))$ by deck transformations on the second component, and in fact
$$\underline{LCC}_{\varepsilon_0,\varepsilon_1}^\bullet(\Lambda,\phi^\epsilon(\Lambda)) \cong LCC_{\varepsilon_0,\varepsilon_1}^\bullet(\Lambda,\phi^\epsilon(\Lambda)) \otimes_{\Z_2} \Z_2[\pi]$$
as vector spaces. The choice of such an isomorphism is equivalent to the choice of a base point in $\mathfrak{p}^{-1}(a_e)$ for every Reeb chord $a$. A similar lifting construction is done also for chords on $\Lambda$.

For a Reeb chord $c$ of $\Lambda$ we define the moduli space ${\mathcal M}_L(c)$
of $J$-holomorphic maps $u \colon (D^2, \partial D^2) \to (\R \times P \times \R, L)$ with one positive puncture asymptotic to $c$. For a $i$-tuple $\mathbf{c}= c_1 \ldots c_i$ of Reeb chords, we define
$${\mathcal M}_L(\mathbf{c}) :=  {\mathcal M}_L(c_1) \times \ldots \times {\mathcal M}_L(c_i).$$
Finally, given Reeb chords $a,b$ from $\phi^\epsilon_{e^t}(\Lambda)$ to $\Lambda$, we define
$$\widetilde{\mathcal M}(a,b) := {\mathcal M}_{\R \times \Lambda,\R \times \phi^\epsilon(\Lambda)}(a; \mathbf{c},b,\mathbf{d}) \times
{\mathcal M}_L(\mathbf{c}) \times {\mathcal M}_{\phi^\epsilon_{e^t}(L)}(\mathbf{d}).$$
Thus the elements of $\widetilde{\mathcal M}(a,b)$ are $J$-holomorphic buildings
$$\tilde{u} = (u, u_{0,1}, \ldots , u_{0, i_0},  u_{1,1}, \ldots , u_{1, i_1}).$$

The boundary of $\underline{LCC}_{\varepsilon_0,\varepsilon_1}^\bullet(\Lambda,\phi^\epsilon(\Lambda))$ counts holomorphic buildings in the moduli spaces $\widetilde{\mathcal M}(a,b)$ as follows.

Given a holomorphic building $\tilde{u} = (u, u_{0,1}, \ldots , u_{0, i_0},  u_{1,1}, \ldots , u_{1, i_1})$, we want to construct a (homotopy class) of a continuous map $l_{\tilde{u},\tilde{a}}:\R \to \widetilde{\overline{L}} \subset \widetilde{L}$ to the universal cover of $\overline{L}$ analogous to the path $\ell_{u, \tilde{p}}$ constructed above. We order the sequence of $i_0+1$ open boundary arcs of $u\in \mathcal {M}_{\R \times \Lambda,\R \times \phi^\epsilon(\Lambda)}(a; \mathbf{c},b,\mathbf{d})$ lying on $\R \times \Lambda$ following the  boundary orientation and starting from the positive puncture, and  denote $l_1,\hdots,l_{i_0+1}$ the corresponding (compact) boundary arcs on $\Lambda$ induced by the canonical projection (here we have used the asymptotic behaviour of $u$). The asymptotic behaviour of each $u_{0,i} \in {\mathcal M}_L(c_i)$ implies that its boundary compactifies to a continuous path $k_i$ on the compactification of $L$ identified with $\overline{L}$. Furthermore, the concatenation $l_1 * k_1 * l_2 * k_2 * \hdots * l_{i_0} *k_{i_0} * l_{i_0+1}$ defines a continuous path on $\overline{L}$ starting on $a_e \in \Lambda = \partial\overline{L}$ and ending on $b_e$, and we define $\ell_{\tilde{u}}$ to be this path. Finally, for $\tilde{a}=(a,x)$ as above, we use
$$\ell_{\tilde{u}, \tilde{a}} \colon \R \to \widetilde{\overline{L}} \subset \widetilde{L}$$
to denote the lift of $l_{\tilde u}$ to the universal cover of $\overline{L}$ satisfying $\lim \limits_{s \to -\infty} \ell_{\tilde{u}, \tilde{a}} = x$.

Given two Reeb chords $a,b$ from $\phi^\epsilon_{e^t}(\Lambda)$ to $\Lambda$ and two lifts $\tilde{a}=(a,x)$ and $\tilde{b}=(b,y)$ we now define
$$\underline{\widetilde{\mathcal M}}(\tilde{a}, \tilde{b}) \subset \widetilde{\mathcal M}(a,b)$$
as the set of holomorphic buildings $\tilde{u} \in  \widetilde{\mathcal M}(a,b)$ such
that $\lim \limits_{s \to + \infty} \ell_{\tilde{u}, \tilde{a}} = y$. We then define the $\Z_2[\pi]$-linear differential
$$\underline{d}_{\varepsilon_0,\varepsilon_1} \colon \underline{LCC}_{\varepsilon_0,\varepsilon_1}^\bullet (\Lambda,\phi^\epsilon(\Lambda)) \to \underline{LCC}_{\varepsilon_0,\varepsilon_1}^{\bullet+1} (\Lambda,\phi^\epsilon(\Lambda))$$
by
$$\underline{d}_{\varepsilon_0,\varepsilon_1} (\tilde{b}) = \sum_{\tilde{a} \in \widetilde{\mathcal{Q}}(\phi^\epsilon(\Lambda),\Lambda)} \# \underline{\widetilde{\mathcal M}}(\tilde{a}, \tilde{b}) \tilde{a}.$$
We denote the corresponding homology by $\underline{LCH}_{\varepsilon_0,\varepsilon_1}^\bullet(\Lambda,\phi^\epsilon(\Lambda))$.

By the now familiar lifting construction we can define a $\Z_2[\pi]$-linear chain map
$$\underline{\delta} \colon \underline{CF}_\bullet(L, \phi^\epsilon_{e^t}(L)) \to \underline{LCC}_{\varepsilon_0,\varepsilon_1}^\bullet(\Lambda,\phi^\epsilon(\Lambda)).$$
 The invariance properties used in Section \ref{wfh} can be extended to the current setting, showing that $\underline{\delta}$ is a $\pi$-equivariant quasi-isomorphism. Therefore we have the following corollary.
\begin{Cor}\label{twisted seidel}
Let $\Lambda\subset P\times\R$ be a Legendrian submanifold admitting an exact Lagrangian filling $L$. Then there is an $\pi_1(L)$-equivariant isomorphism
$$H_{i}(\tilde{L}; \Z_{2}) \simeq \underline{LCH}_{\varepsilon_0,\varepsilon_1}^\bullet(\Lambda,\phi^\epsilon(\Lambda))$$
of $\Z_{2}[\pi]$-modules, where $\varepsilon_0,\varepsilon_1$ are the augmentations induced by $L$ and $\phi^\epsilon_{e^t}(L)$, respectively, and where all gradings are taken modulo the Maslov number of $L$.
\end{Cor}
\begin{Prop}\label{sec:high-dimens-gener}
Any exact Lagrangian filling $L$ of a simply connected Legendrian submanifold $\Lambda \subset (\R^{2n+1}, \xi_{st})$ having a single transverse Reeb chord is simply connected.
\end{Prop}
\begin{proof}
Let $\widetilde{L}$ be the universal cover of $L$ and $\widetilde{\Lambda}$ the preimage of $\Lambda$ in $\widetilde{L}$. Since $\Lambda$ is simply connected, $\widetilde{\Lambda}$ consists of a disjoint union of spheres and $\pi = \pi_1(L)$ acts freely and transitively on $\pi_0(\widetilde{\Lambda})$.

Since $\Lambda$ has a unique Reeb chord, the differential $\underline{d}_{\varepsilon_0,\varepsilon_1}$ is trivial.  It follows that $\underline{LCH}_{\varepsilon_0,\varepsilon_1}(\Lambda,\phi^\epsilon(\Lambda)) \cong \Z_{2}[\pi]$ and, moreover, the action of $\pi$ on this homology group is fixed-point free. On the other hand $H_0(\widetilde{L}; \Z_{2}) \cong \Z_{2}$ because $\widetilde{L}$ is connected. Since $\pi$ acts trivially on $H_0(\widetilde{L};\Z_2)$, Corollary \ref{twisted seidel} thus implies that $\Z_{2}[\pi] \cong \Z_{2}$, which is possible only if $\pi$ is the trivial group.
\end{proof}
\begin{proof}[Proof of Theorem \ref{thm:fukayaseidelsmith}]
Theorem \ref{malgthinstregbam} implies that $L$ is a $\Z_2$-homology disc and thus, in particular, orientable and spin. Since we did not discuss orientation issues for moduli spaces here, we refer to the version  of Theorem \ref{malgthinstregbam} for integer coefficients in \cite[Theorem 1.4]{OnHolRegFlexEndoc}, which implies that $L$ in fact is a $\Z$-homology disc. Proposition \ref{sec:high-dimens-gener} can now be applied, showing that $L$ moreover is simply connected. In conclusion, we have shown that $L$ is contractible.

Assume now that there exists an exact filling $L$ of $\Lambda$. Since $L$, and hence $\Lambda$, is orientable and has vanishing Maslov class by the above, Theorem \ref{malgthinstregbam} shows that the unique Reeb chord of $\Lambda$ is of degree $n$. We can thus apply \cite[Theorem 1.1]{EkholmSmith} to show that $\Lambda$ is a homotopy sphere.

For $n \ne 3,4$, it follows that $(\overline{L},\partial\overline{L})$ is diffeomorphic to $(D^{n+1},\partial D^{n+1})$ and, in particular, $\Lambda$ is thus the standard sphere. For $n=1$ this follows by elementary methods while, in high dimensions, this follows from Smale's h-cobordism theorem \cite{hcob}. In the case $n=2$ the same is true due to Perelman's positive answer to the Poincar\'{e} conjecture \cite{PoincareConjecturePerelman}.
The smooth Poincar\'{e} conjecture for smooth homotopy 5-spheres, as proven by  \cite{KervaireMilnor}, shows that this also holds when $n=4$ under the additional assumption that $\Lambda$ is diffeomorphic to the standard sphere.
When $n=3$, $\Lambda$ is diffeomorphic to $S^3$ by \cite{PoincareConjecturePerelman} and  $(\overline{L},\partial\overline{L})$ is  homeomorphic to $(D^{4},\partial D^{4})$ because any $4$-dimensional homotopy sphere is {\em homeomorphic} to $S^4$ by Freidman's theorem \cite{Freedman4dimHomeo}.

\end{proof}

\begin{proof}[Alternative proof of Theorem \ref{thm:fukayaseidelsmith}]
The above theorem can also be proven by combining \cite[Theorem 1.1]{EkholmSmith} with Seidel's isomorphism. Observe that the former theorem is proven using linearised Legendrian contact homology with local coefficients. First, observe that Theorem \ref{malgthinstregbam} implies that $L$ is a $\Z_2$-homology disc and, hence, both $\Lambda$ and $L$ are oriented and spin. A version of Seidel's isomorphism for coefficients in $\Z$ (see e.g.~\cite[Theorem 1.4]{OnHolRegFlexEndoc}) shows that $L$ is a $\Z$-homology disc. It follows that both $L$ and $\Lambda$ have vanishing Maslov classes and that the Reeb chord of $\Lambda$ is in degree $n$. \cite[Theorem 1.1]{EkholmSmith} can now be applied, showing that $\Lambda$ is a homotopy sphere.

In order to conclude we need the following lemma.
\begin{Lem}\label{sec:high-dimens-gener-1}
  Let $L$ be an exact Lagrangian filling of a Legendrian submanifold $\Lambda\subset \R^{2n+1}$, and let $\overline{L}=L\cap (-\infty,-M]\times \R^{2n+1}$ be its compact part, where $M>0$. Then there exists a Lagrangian immersion $L_2:=\overline{L}\cup_\Lambda\overline{L}\hookrightarrow \R\times\R^{2n+1}$ such that $L_2\cap  (-\infty,-M]\times \R^{2n+1}=\overline{L}$. Moreover, there is a natural bijective correspondence between the Reeb chords on $\Lambda$ and the double-points of $L_2$ which increases the grading by one (the Reeb chords on $\Lambda$ are generic if and only if the double-points of $L_2$ are transverse).
\end{Lem}

\begin{proof}
  Let $\phi:\R\times\R^{2n+1}\rightarrow \R\times\R^{2n+1}$ be the symplectomorphism defined by $(t,\mathbf{x},\mathbf{y},z) \mapsto (-t,\mathbf{x},e^{2t}\mathbf{y},-e^{2t}z)$. We denote by $\overline{L}'$ the image of $\overline{L}$ by $\phi$. It lies in $[M,\infty)\times\R^{2n+1}$. In order to get $L_2$, we connect $\overline{L}$ and $\overline{L}'$ by a Lagrangian immersion of $[-M-\varepsilon,M+\varepsilon]\times\Lambda$ constructed in the following way. Let $\rho:[e^{-M-\varepsilon},e^{M+\varepsilon}]\rightarrow \mathbb{R}$ be a smooth function satisfying the following properties:
  \begin{enumerate}
  \item $\rho(t)=1$ for $t\leq e^{-M}$;
\item $\rho(t)=-\frac{1}{t^2}$ for $t\geq e^M$;
\item $\rho(t)$ vanishes exactly at $t=t_0$, where $\rho'(t_0)\not=0$;
\item $\rho(t)$ admits a primitive $f(t) > 0$ satisfying $f(t)=t$ for $t\leq e^{-M}$, and $f(t)=\frac{1}{t}$ for $t\geq e^M$.
  \end{enumerate}

Use $\Lambda(\mathbf{s})=(\mathbf{x}(\mathbf{s}),\mathbf{y}(\mathbf{s}),z(\mathbf{s}))$ to denote a parametrisation of $\Lambda$. The required immersion is now given by $(t,\mathbf{s})\rightarrow (t,\mathbf{x}(\mathbf{s}),e^{-t}f(e^t)\mathbf{y}(\mathbf{s}),\rho(e^t)z(\mathbf{s}))$. Note that since $\Lambda$ is an embedding, the double points of this cylinder are all contained in the unique level-set $\{ t=t_0\}$ where $\rho(t_0)=0$. Moreover, there is a bijective correspondence between these double points and the double points of the Lagrangian projection $(\mathbf{x}(\mathbf{s}),\mathbf{y}(\mathbf{s}))$ which, in turn, correspond bijectively to the Reeb chord of $\Lambda$.
\end{proof}

Now we conclude the proof of Theorem \ref{thm:fukayaseidelsmith}. Lemma \ref{sec:high-dimens-gener-1} applied to the filling $L$ produces an exact Lagrangian immersion inside $\R^{2n+2}$ of a closed manifold having a single transverse double-point of degree $n+1$. An application of \cite[Theorem 1.1]{EkholmSmith} shows that the latter manifold is a homotopy-sphere, which finally shows that $L$ is contractible.
\end{proof}

We conclude this section with the proof of Corollary \ref{sec:introduction-cor}:

\begin{proof}[Proof of Corollary \ref{sec:introduction-cor}] Let $W$ denote a filling of $\Lambda^-$, which can be seen as the concatenation of the standard filling $L_0$ of $\Lambda_0$ and an exact Lagrangian cobordism $V$ from $\Lambda_0$ to $\Lambda^-$. Theorem \ref{thm:fukayaseidelsmith} readily implies that the concatenation $V \odot L$ of $V$ and $L$ is a concordance. Here one must use the h-cobordism theorem in the case $n>4$ and the solution of the (generalised) Poincar\'{e} conjecture in the cases $n=3,4$ (see the proof of Theorem \ref{thm:fukayaseidelsmith}). Theorem \ref{thm:main} applied to $V \odot L$ implies that the map induced in linearised LCH is an isomorphism. Using the fact that the DGAs of both $\Lambda^0$ and $\Lambda^+$ have a single generator and hence vanishing differential , we conclude that
$$ \Phi_{V \odot L}=\Phi_V \circ \Phi_L $$
is an isomorphism of DGAs.

In the case when $\Lambda^-$ satisfies the same assumptions as $\Lambda^-$, the argument above shows that $L$ is a concordance and that $\Phi_L$ is an isomorphism.
\end{proof}

\section{Non-symmetry for Lagrangian concordances in high dimensions}
\label{sec:nonsymm}

\begin{figure}[hftp]
\labellist \small\hair 2pt \pinlabel {$a_1$} [bl] at 515 173
\pinlabel {$a_2$} [bl] at 531 76 \pinlabel {$a_3$} [bl] at 397 144
\pinlabel {$b_6$} [bl] at 482 150 \pinlabel {$a_4$} [bl] at 366 175
\pinlabel {$b_5$} [bl] at 362 115 \pinlabel {$b_3$} [bl] at 338 199
\pinlabel {$b_4$} [bl] at 326 145 \pinlabel {$c_1$} [bl] at 303 168
\pinlabel {$c_2$} [bl] at 258 40 \pinlabel {$a_5$} [bl] at 207 180
\pinlabel {$b_1$} [bl] at 158 230 \pinlabel {$b_2$} [bl] at 116 138
\endlabellist
  \centering
  \includegraphics[height=3.5cm]{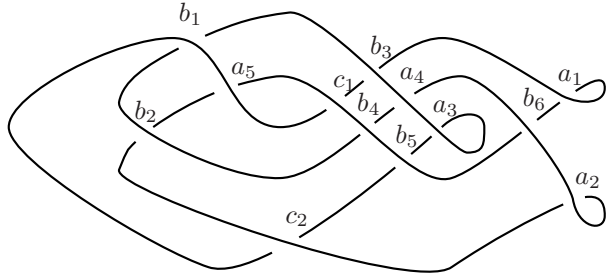}
  \caption{The Lagrangian projection of $\Lambda_{m(9_{46})}$, cf.~Figure 5 in \cite{LagrConcordNotASymmRel}.}
  \label{fig:LCHlag}
\end{figure}

Let $\Lambda_0 \subset (\R^3, \xi_{st})$ be the Legendrian unknot with
Thurston-Bennequin invariant ${\tt tb}(\Lambda_0)=-1$ and $\Lambda =
\Lambda_{m(9_{46})}\subset (\R^3, \xi_{st})$ the Legendrian
representative of the knot $9_{46}$ whose Lagrangian projection is
described in Figure~\ref{fig:LCHlag}. The chords on $\Lambda$ are
denoted by $a_i$, $b_j$ and $c_k$, where $i=1,\dots,5$, $j=1,\dots,
6$ and $k=1,2$.

In \cite{LagrConcordNotASymmRel},  Chantraine exhibited a Lagrangian concordance from $\Lambda_0$ to $\Lambda_{m(9_{46})}$ and proved that a Lagrangian concordance from $\Lambda_{m(9_{46})}$ to $\Lambda_0$ cannot exist. This
example showed that Lagrangian concordance is not a symmetric relation in dimension
$3$. In this section, we will apply the front spinning construction of Section \ref{sec:front-spinn-constr} to Chantraine's examples to produce new examples of the non-symmetry of the Lagrangian concordance relation in high dimensions.

\begin{Rem}
From the results of Eliahsberg-Murphy~\cite{LagCaps} it follows that
there are many ``flexible examples'' of non-invertible exact
Lagrangian concordances, but none of those examples has a fillable
negative end. The examples we provide 
are ``rigid''
non-invertible exact Lagrangian concordances.
\end{Rem}

\begin{Thm}\label{thm:onespintooldexample}
There exists a Lagrangian concordance from $\Sigma_{S^m}\Lambda_0$ to $\Sigma_{S^m} \Lambda$. However, there is no exact
Lagrangian concordance from $\Sigma_{S^m} \Lambda$ to $\Sigma_{S^m}
\Lambda_0$.
\end{Thm}
\begin{proof}
Consider $\Sigma_{S^m} \Lambda_0$ and $\Sigma_{S^m} \Lambda$. Since
the first author in \cite{LagrConcordNotASymmRel} has constructed an
exact Lagrangian concordance $C$ from $\Lambda_0$ to $\Lambda$,
\cite[Proposition 1.1]{NoteOnFrontSpinning} implies that there
exists a Lagrangian concordance $\Sigma_{S^m} C$
from $\Sigma_{S^m} \Lambda_0$ to $\Sigma_{S^m} \Lambda$.
Now we prove that there is no exact Lagrangian concordance from
$\Sigma_{S^m}\Lambda$ to $\Sigma_{S^m}\Lambda_0$.

First, observe that $\Sigma_{S^m}\Lambda_0$ has an exact Lagrangian filling diffeomorphic to $D^2 \times S^m$ which is obtained by spinning the exact Lagrangian disc filling $L_{\Lambda_0}$ of $\Lambda_0$, and hence $(\mathcal A(\Sigma_{S^m}\Lambda_0),\partial_{\Sigma_{S^m}\Lambda_0})$ admits an augmentation.
After an explicit Legendrian perturbation of $\Sigma_{S^m}\Lambda_0$, see e.g.~\cite[Section 3.1]{EstimNumbrReebChordLinReprCharAlg}, the resulting Legendrian submanifold has exactly two transverse Reeb chords, one in degree $1$ and one in degree $1+m$. In particular, given any pair $\varepsilon'_0$, $\varepsilon'_1$ of augmentations, we have $LCH^i_{\varepsilon'_0, \varepsilon'_1} (\Sigma_{S^m}\Lambda_0)=0$ whenever $i < 0$.

Then, note that
$(\mathcal A(\Lambda),\partial_\Lambda)$ admits
several augmentations and at least two of those, namely
$\varepsilon_0$ and $\varepsilon_1$ which were discussed in the proof of \cite[Theorem 1.1]{LagrConcordNotASymmRel}, are induced by
exact Lagrangian fillings of $\Lambda$.
We call these fillings $L_0$ and $L_1$ respectively.
As shown in \cite{LagrConcordNotASymmRel},
$L_{\Lambda}^{\varepsilon_0}=L_{\Lambda_0}\odot C$, while
$L_{\Lambda}^{\varepsilon_1}$ is described in Figure~\ref{fig:notriv}
where the last move consist of capping the two components with
a Lagrangian disc filling of $\Lambda_0$.

\begin{figure}[hftp]
  \centering
  \includegraphics[height=9cm]{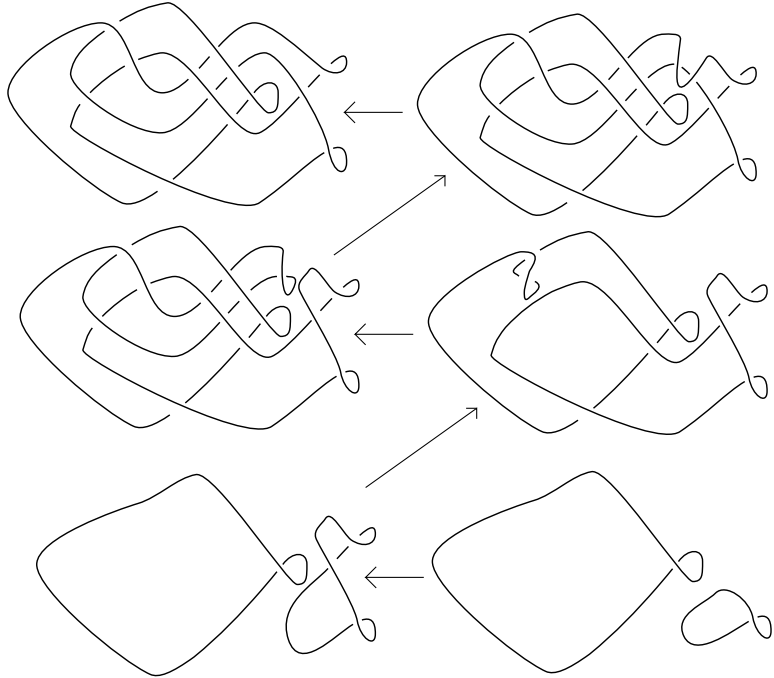}
  \caption{The exact Lagrangian filling
$L_{\Lambda}^{\varepsilon_1}$ decomposed into elementary Lagrangian
handle attachments, where the last two correspond to capping off the
unknots described in the lower right corner.}\label{fig:notriv}
\end{figure}

Those qualitative moves of Lagrangian
projections are obtain from moves on fronts diagram by the process explained in the beginning of Section 3 in \cite{LagrConcordNotASymmRel} and thus represents a sequence of elementary Lagrangian cobordisms. Also, we here rely heavily on \cite{LegKnotsLagCob} for the computations
of the induced augmentations. Recall that $|a_i|=1$, $|b_j|=0$ and $|c_k|=-1$ for
$i=1,\dots,5$, $j=1,\dots, 6$ and $k=1,2$. The
  augmentations $\varepsilon_i$, $i=0,1$, are described as
  follows:
\begin{itemize}
  \item $\varepsilon_0(b_{2})=\varepsilon_0(b_4)=\varepsilon_0(b_5)=1$ and $\varepsilon_0$ vanishes on all other generators;
  \item $\varepsilon_1(b_1)=\varepsilon_1(b_3)=\varepsilon_1(b_6)=1$ and
  $\varepsilon_1$ vanishes on all other generators.
\end{itemize}

Following  \cite[proof of Theorem 1.1]{LagrConcordNotASymmRel}, we observe that
 $LCH_{\varepsilon_0, \varepsilon_1}^{-1}(\Lambda)\neq 0$.
From \cite[Proposition 1.1]{NoteOnFrontSpinning}
it follows that there exist exact
Lagrangian fillings $\Sigma_{S^m}L_0$ and $\Sigma_{S^m}L_1$ of
$\Sigma_{S^m}\Lambda$ which induce augmentations $\tilde{\varepsilon}_0$ and $\tilde{\varepsilon}_1$ respectively.

Now assume by contradiction that there exists a Lagrangian concordance
$C'$ from $\Sigma_{S^m}\Lambda$ to $\Sigma_{S^m}\Lambda_0$, and let $\Phi_{C'}:\mathcal A(\Sigma_{S^m}\Lambda)\to \mathcal A(\Sigma_{S^m}\Lambda_0)$ be the induced map between the Chekanov-Eliashberg algebras. Then Theorem~\ref{thm:main} implies
that $LCH_{\tilde{\varepsilon}_0,\tilde{\varepsilon}_1}^{\bullet}(\Sigma_{S^m}\Lambda)$ is isomorphic to $LCH_{\tilde{\varepsilon}_0\circ\Phi_{C'},\tilde{\varepsilon}_1\circ\Phi_{C'}}^{\bullet}(\Sigma_{S^m}\Lambda_0)$. The K\"unneth-type
formula (Theorem~\ref{kunnethformulaspinbilinearised}) implies that
$$(LCH_{\tilde{\varepsilon}_0,\tilde{\varepsilon}_1}^{\bullet}(\Sigma_{S^m}\Lambda))\simeq (LCH_{\varepsilon_0,\varepsilon_1}^{\bullet}(\Lambda)) \otimes (H_\bullet(S^m)).$$

In particular, $LCH_{\tilde{\varepsilon}_0,\tilde{\varepsilon}_1}^{-1}(\Sigma_{S^m}\Lambda)\neq 0$ because $LCH_{\varepsilon_0,\varepsilon_1}^{-1}(\Lambda)\neq 0$. This gives a
contradiction because, as we have already mentioned, $LCH_{\varepsilon'_0,\varepsilon'_1}^{i}(\Sigma_{S^m}\Lambda_0)=0$ for any pair of augmentations $\varepsilon'_0$, $\varepsilon'_1$ whenever $i < 0$.

\end{proof}
By making repeated $S^m$-spins of $\Lambda$ and $\Lambda_0$, the
proof of the above theorem generalises to the proof of the following
statement.
\begin{Prop}\label{prop:manyspinstooldexample}
For every $m_1,\dots, m_k\in \mathbb N$ there exist fillable Legendrian submanifolds
$\Lambda_1, \Lambda_2\subset (\R^{2(1+\sum_i m_i)+1}, \xi_{st})$
diffeomorphic to $S^1\times S^{m_1}\times \dots \times S^{m_k}$ with
the property that
\begin{itemize}
\item there is a Lagrangian concordance from $\Lambda_1$ to $\Lambda_2$;
\item there is no Lagrangian concordance from $\Lambda_2$ to
$\Lambda_1$.
\end{itemize}
\end{Prop}

\bibliographystyle{alphanum}
\bibliography{references}

\end{document}